\documentclass[12pt,reqno]{amsart}

\usepackage{amsmath}
\usepackage{amsthm}
\usepackage{amssymb}
\usepackage{amsfonts}
\usepackage{amsxtra}
\usepackage{fullpage}
\usepackage{amscd}
\usepackage{color} 
\usepackage{bm} 
\usepackage{mathrsfs}
\usepackage{subfigure}
\usepackage{tikz}
\usepackage{tikz-cd}
\usetikzlibrary{cd}
\usepackage{enumerate}
\usepackage[all]{xy}
\usepackage[mathscr]{eucal}
\usepackage{verbatim}
\usepackage{hyperref}

\let\nc\newcommand
\let\renc\renewcommand

\theoremstyle{plain}
\newtheorem{thm}{Theorem}
\newtheorem{prop}[thm]{Proposition}
\newtheorem{cor}[thm]{Corollary}
\newtheorem*{thm*}{Theorem}
\newtheorem*{prop*}{Proposition}
\newtheorem*{cor*}{Corollary}
\newtheorem{lem}[thm]{Lemma}

\theoremstyle{definition}
\newtheorem{defn}[thm]{Definition}

\newtheorem{remark}[thm]{Remark}

\numberwithin{thm}{section}

\nc{\bdm}{\begin{displaymath}}
\nc{\edm}{\end{displaymath}}
\nc{\bthm}{\begin{thm}}
	\nc{\ethm}{\end{thm}}
\nc{\blem}{\begin{lem}}
	\nc{\elem}{\end{lem}}
\nc{\bcor}{\begin{cor}}
	\nc{\ecor}{\end{cor}}
\nc{\bprop}{\begin{prop}}
	\nc{\eprop}{\end{prop}}
\nc{\bdef}{\begin{defn}}
	\nc{\eddef}{\end{defn}}

\makeatletter
\renewcommand{\subsection}{\@startsection{subsection}{2}{0pt}{-3ex
		plus -1ex minus -0.2ex}{-2mm plus -0pt minus
		-2pt}{\normalfont\bfseries}} \makeatother

\numberwithin{equation}{section}

\newcommand{\Lmod}[1]{#1\text{-}{\mathsf{mod}}}

\newcommand{\idot}{{\:\raisebox{2pt}{\text{\circle*{1.5}}}}}

\DeclareMathOperator{\Ext}{\mathrm{Ext}}

\DeclareMathOperator{\Ker}{\mathrm{Ker}}

\DeclareMathOperator{\End}{\mathrm{End}}

\DeclareMathOperator{\gr}{\mathrm{gr}}

\newcommand{\beq}{\begin{equation}\label}
\newcommand{\eeq}{\end{equation}}

\DeclareMathOperator{\Spec}{\mathrm{Spec}}

\newcommand{\iso}{{\;\stackrel{_\sim}{\to}\;}}

\DeclareMathOperator{\Hom}{\mathrm{Hom}}
\DeclareMathOperator{\GL}{\mathrm{GL}}

\nc{\Z}{\mathbb{Z}}
\newcommand{\N}{\mathbb{N}}
\newcommand{\Q}{\mathbb{Q}}

\newcommand{\C}{\mathbb{C}}
\newcommand{\h}{\mathfrak{h}}

\newcommand{\eu}{\mathrm{eu}}

\nc{\rank}{\textrm{rank} \,}
\nc{\ds}{\dots}

\let\mc\mathcal
\let\mf\mathfrak

\nc{\mbf}{\mathbf}
\nc{\Res}{\mathsf{Res} \, }
\nc{\Ind}{\mathsf{Ind} \, }
\nc{\cont}{\textrm{cont}}
\renewcommand{\mod}{\textrm{mod}}
\nc{\msf}{\mathsf}

\nc{\minusone}{-1}
\nc{\minustwo}{-2}
\nc{\Mod}{\mathrm{Mod} \,}
\nc{\ms}{\mathscr}
\nc{\Frac}{\mathrm{Frac} \,}
\nc{\ra}{\rightarrow}
\nc{\hra}{\hookrightarrow}
\nc{\lab}{\label}
\renc{\O}{\mc{O}}
\nc{\Tan}{\mc{T}}
\nc{\ul}{\underline}
\nc{\s}{\mathfrak{S}}
\nc{\g}{\mf{g}}
\nc{\pa}{\partial}
\nc{\tit}{\textit}
\nc{\Maxspec}{\mathrm{Maxspec} \, }
\nc{\gldim}{\mathrm{gl.dim}}
\nc{\rkm}{\mathrm{rk} \, (\mf{m})}
\nc{\sm}{\mathrm{sm}}
\nc{\PD}{\mathbb{PD}}
\nc{\hilb}{\textrm{Hilb}}
\nc{\<}{\langle}
\renc{\>}{\rangle}
\nc{\T}{\mathrm{T}}
\nc{\X}{\mathbb{X}}
\nc{\F}{\mathbb{F}}
\nc{\id}{\msf{id}}
\nc{\A}{\mathbb{A}}
\nc{\Grat}{\mc{Grat}}
\nc{\Squo}[1]{\A^{(#1)}}
\nc{\twist}{\mathrm{twist}}
\nc{\Cd}{\mc{C}}
\nc{\Span}{\mathrm{Span}}
\nc{\Grass}{\mathrm{Gr}}
\nc{\Supp}{\mathrm{Supp}}
\nc{\Irr}{\mathrm{Irr}}
\renc{\o}{\otimes}
\nc{\fin}{\mathrm{fin}}
\nc{\aff}{\mathrm{aff}}
\nc{\algD}{\mf{D}}
\nc{\hr}{\mf{h}_{\textrm{reg}}}
\nc{\D}{\mathscr{D}}
\nc{\PIdeg}{\mathrm{PI-degree}}
\nc{\ch}{\mathrm{ch}}
\nc{\ev}{\mathsf{ev}}
\nc{\Stab}{\mathrm{Stab}}
\nc{\Der}{\mathrm{Der}}
\nc{\rightsim}{\stackrel{\sim}{\longrightarrow}}
\nc{\HZ}{H_{\mbf{h},\Z}(\Z_m)}
\nc{\sing}{\mathrm{sing}}
\nc{\dd}{\mathscr{D}}

\nc{\vc}{\underline{\mathbf{c}}}
\nc{\ba}{\mathbf{a}}
\nc{\reg}{\mathrm{reg}}
\nc{\Amp}{\mathrm{Amp}}
\nc{\Nef}{\mathrm{Nef}}
\nc{\SL}{\mathrm{SL}}
\nc{\Sp}{\mathrm{Sp}}
\nc{\Sym}{\mathrm{Sym}}
\nc{\Mov}{\mathrm{Mov}}
\nc{\Pic}{\mathrm{Pic}}
\nc{\Cs}{\C^{\times}}
\nc{\Nak}[3]{\mf{M}_{{#1}} ({#2},{#3}) }
\nc{\Naka}[2]{\mf{M}({#1},{#2}) }
\nc{\Mtheta}[1]{\mc{M}_{#1}}

\nc{\bw}{\mathbf{w}}

\nc{\bn}{\mathbf{n}}
\nc{\CB}{\mathrm{CB}}
\nc{\GVect}{\Lambda}
\nc{\pZ}{\overline{Z}}
\nc{\Tang}{\mc{T}}
\nc{\K}{\mathbb{K}}
\nc{\Id}{\mathrm{Id}}
\nc{\be}{\boldsymbol{e}}
\nc{\GKdim}{\mr{GK}\textrm{-}\mr{dim}}

\newcommand{\mr}{\mathrm}

\newcommand{\git}{\ensuremath{/\!\!/\!}}
\nc{\ZHgl}{{Z(H_{\bc})_{\mf{gl}(n)}}}
\nc{\G}{G}
\nc{\eo}{\mathbf{a}}
\nc{\ezero}{\mathbf{a}_0}
\nc{\ei}{\mathbf{a}_i}
\nc{\elminusone}{\mathbf{a}_{\ell-1}}

\nc{\red}[1]{\textcolor{red}{#1}}

\begin{document}
	
	\title{Two invariant subalgebras of rational Cherednik algebras}
	
	\author[G. Bellamy]{Gwyn Bellamy}

    \author[M. Feigin]{Misha Feigin}

    \author[N. Hird]{Niall Hird}

	\address{School of Mathematics and Statistics, University of Glasgow, University Place,
		Glasgow, G12 8QQ.}
	\email{gwyn.bellamy@glasgow.ac.uk}
        \email{misha.feigin@glasgow.ac.uk}
        

	\begin{abstract}
    Originally motivated by connections to integrable systems, two  natural subalgebras of the rational Cherednik algebra have been considered in the literature.  The first 
    is the subalgebra of all degree zero elements and the second 
    is the Dunkl angular momentum subalgebra. 

    In this article, we study the ring-theoretic and homological properties of these algebras. 
    Our approach is to realise them as rings of invariants under the action of certain reductive subgroups of $\SL_2$. This allows us to describe their centres. 
Moreover, we show that they are Auslander--Gorenstein and Cohen--Macaulay and, at $t = 0$, give rise to prime PI-algebras whose PI-degree we compute. 
    
    Since the degree zero subalgebra 
    can be realized as the ring of invariants for the maximal torus $\T \subset \SL_2$ and the action of this torus on the rational Cherednik algebra is Hamiltonian, we also consider its (quantum) Hamiltonian reduction with respect to $\T$. At $t = 1$, the quantum Hamiltonian reduction of the spherical subalgebra is a filtered quantization of the quotient of the minimal nilpotent orbit closure $\overline{\mc{O}}_{\min}$ in $\mf{gl}(n)$ by the reflection group $W$. At $t = 0$, we get a graded Poisson deformation of the symplectic singularity $\overline{\mc{O}}_{\min}/W$.
	\end{abstract}
	
	\maketitle

\section{Introduction}\label{sec1}

One key motivation for the introduction of rational Cherednik algebras $H_{t,c}$ by Etingof and Ginzburg \cite{EG} was their connection with the generalized Calogero--Moser integrable systems. For instance, the algebras can be used to give a uniform proof of the complete integrability of these systems \cite{Heckman}.
In the case of a real reflection group $W$ and parameter $t=1$, the second author together with Hakobyan introduced in \cite{Feiginangular} the \textit{Dunkl angular momentum algebra} $H_{t,c}^{\mathfrak{so}(n)}$ as the subalgebra of a rational Cherednik algebra   generated by the (Dunkl deformed) angular momentum operators $x_i y_j - x_j y_i$ and the reflection group $W$. Part of the motivation for introducing this algebra is the fact that it contains  in its centre essentially the angular generalized Calogero--Moser  Hamiltonian. This makes the angular part of generalized Calogero--Moser systems amenable to study via representation theory. At the same time,  the closely related \textit{degree zero} subalgebra $H_{t,c}^{\mf{gl}(n)}$ was also introduced in \cite{Feiginangular}. This algebra is generated by the elements $x_i y_j$ and the reflection group, and it is related, in turn, to Calogero--Moser systems in the harmonic confinement.

Another interesting context where the algebras $H_{t,c}^{\mf{gl}(n)}$ and $H_{t,c}^{\mf{so}(n)}$ arise is deformed Howe duality. The algebra $H_{t,c}^{\mf{so}(n)}$ plays a role in the Dunkl deformed Howe dual pair $(\mf{so}(n),\mf{sl}(2))$ \cite{CiubotaruMartino} while the algebra $H_{t,c}^{\mf{gl}(n)}$ relates to the Dunkl deformation of $(\mf{gl}(n), \mf{gl}(1))$ duality \cite{FeiginVrabec}.

In this article, we study the ring-theoretic and homological properties of these two subalgebras. Our primary focus is on understanding their properties in the case $t = 0$. Since the algebras are defined in terms of a set of generators, it can be difficult to prove directly any of their abstract (e.g. homological) properties. Our key observation is that both of these algebras can be realised as rings of invariants under certain groups of automorphisms of the rational Cherednik algebra $H_{t,c}$. 

Let $\T$ be the maximal torus of diagonal matrices in $\SL_2$. Then $H_{t,c}^{\mf{gl}(n)} = H_{t,c}^{\T}$. In the case of a real reflection group $W$ the entire group $\SL_2$ acts by automorphisms on the rational Cherednik algebra $H_{t,c}$ \cite{EG}. In this case we show  in Theorem~\ref{sonissl2} that $H_{t,c}^{\mathfrak{so}(n)} = H_{t,c}^{\SL_2}$.


Realizing both algebras as rings of invariants under a connected reductive group $\Gamma \subset \SL_2$ allows us to give a remarkably uniform treatment of these algebras; this we do in Sections~\ref{sec:invariantsub} and   \ref{sec:invariantt=0}. 
For $\Gamma=SL_2$ we assume that $W$ is a real reflection group.

\subsection*{The centre of the algebras}

The rational Cherednik algebra $H_{t,c}$ has a big centre at $t=0$ \cite{EG}. Our original motivation was to describe the centre of the algebras $H_{0,c}^{\mf{gl}(n)}$ and $H_{0,c}^{\mf{so}(n)}$. Let $P = \C[\h \times \h^*]$, where $\h$ is the reflection representation of $W$ and $\h^*$ is its dual. Recall that $H_{0,0}$ is the skew group ring $P \rtimes W$. As above, $\Gamma$ can be either $\T$ or $\SL_2$. Let $Z(A)$ be the centre of an algebra $A$, and let $\gr A$ be the associated graded algebra of a filtred algebra $A$.

\begin{thm*}(\ref{assgrade})
    For any $c$, $\gr Z(H_{0,c}^{\Gamma}) = Z(P^{\Gamma} \rtimes W)$. 
\end{thm*}


Let us assume that the group $W$ is irreducible. Then the centre $Z(W)$ is a cyclic group $\Z/\ell \Z$. Let $\eo_0, \ds, \eo_{\ell-1}$ be the complete set of primitive idempotents in $\C Z(W)$.

Theorem \ref{assgrade}, together with the simple fact that if $A\subset B$ and $\gr A=\gr B$ then $A=B$, leads to the following 
corollary. 

\begin{cor*}(\ref{centregln})
There is equality of commutative rings
\[
Z(H_{0,c}^{\Gamma}) = Z(H_{0,c})^{\Gamma} \o \C Z(W),
\]
where  $\C Z(W)$ is the group algebra of the centre $Z(W)$ of the group $W$. 
In particular, as a $Z(H_{0,c})^{\Gamma}$-module, 
\[
Z(H_{0,c}^{\Gamma}) = \bigoplus_{i = 0}^{\ell-1} Z(H_{0,c})^{\Gamma} \ei.
\]
\end{cor*}

If we set $\eo = \eo_0$ to be the symmetrising idempotent then it follows that the centre of $\eo H_{0,c}^{\Gamma} \eo$ equals $\eo Z(H_{0,c}^{\Gamma}) \cong Z(H_{0,c})^{\Gamma}$. In fact, each summand $Z(H_{0,c})^{\Gamma} \ei$ of $Z(H_{0,c}^{\Gamma})$ is isomorphic to $Z(H_{0,c})^{\Gamma}$. Being a ring of invariants, the latter is very well-behaved (despite being singular), as illustrated by the following result. 

\begin{thm*}(\ref{ratsing})
	The affine scheme $Y_{c} := \Spec Z(H_{0,c})^{\Gamma}$ is a normal irreducible variety. It is Gorenstein with rational singularities.
\end{thm*}


\subsection*{Ring-theoretic and homological properties}

Since the idempotents $\ei$ are central in $H_{t,c}^{\Gamma}$, there is a decomposition of algebras $H_{t,c}^{\Gamma} = \bigoplus_{i = 0}^{\ell-1} \ei H_{t,c}^{\Gamma} \ei$. If $\overline{W}$ is the quotient $W / Z(W)$ then associated to each idempotent $\ei$ is a twisted group algebra $\C_{f_i} \overline{W}$ for a corresponding 2-cocycle $f_i$, and we show that the associated graded of $\ei H_{t,c}^{\Gamma} \ei$ equals $P^{\Gamma} \rtimes \C_{f_i} \overline{W}$. This implies that each of the rings $\ei H_{t,c}^{\Gamma} \ei$ is prime. Our main result on the general properties of these algebras is summarized below.  

\begin{thm*}(\ref{thm:AuslanderGorensteinH0})
	The algebras $H_{t,c}^{\Gamma}$ and $\ei H_{t,c}^{\Gamma} \ei$ are Auslander--Gorenstein and (GK)  Cohen--Macaulay. These algebras have Gelfand--Kirillov dimension $2 \dim \h - \dim \Gamma$. 
\end{thm*}

Our main tool is to compute the associated graded of these algebras, and their centres and lift ring-theoretic and homological properties from the associated graded. Our approach follows closely that of Brown and Changtong \cite{BrownChangtong}.

When $t = 0$, Corollary~\ref{centregln} implies that both $H_{0,c}^{\Gamma}$ and $\eo H_{0,c}^{\Gamma} \eo$ are finite modules over their respective centres. In particular, they are Polynomial Identity (PI) rings. We study the simple modules over the prime PI ring $\eo H_{0,c}^{\Gamma} \eo$. Our main result is summarized as follows. 

\begin{thm*}(\ref{thm:PIsimplemodules})
	For all parameters $c$, 
	\begin{enumerate}
		\item[(i)] $H_{0, c}^{\Gamma}$ and $\eo H_{0, c}^{\Gamma}\eo$ are PI algebras and the PI degree of the prime PI-algebra $\eo H_{0, c}^{\Gamma}\eo$ equals $|\overline{W}|$.
		\item[(ii)] Let $L$ be a simple $\eo H_{0, c}^{\Gamma}\eo$-module whose support is contained in the regular locus of $Y_{c}$. Then $L |_{\overline{W}} \cong \C \overline{W}$. 
		\item[(iii)] The regular locus of $Y_{c}$ is contained in the Azumaya locus of $\eo H_{0, c}^{\Gamma}\eo$. 
	\end{enumerate}
\end{thm*}

We do not know if the two loci in Theorem~\ref{thm:PIsimplemodules}(iii) are actually equal. 

We also examine the centre of the degree zero subalgebra for $t\neq 0$. Since the centre of the rational Cherednik algebra is trivial in this case, one would expect that the centre of $H_{t,c}^{\mf{gl}(n)}$ is also small when $t\neq 0$. This is indeed the case (cf. \cite{Feiginangular}, where one central generator was identified in the case of real reflection groups).


\begin{prop*}(\ref{centrefortnot})
For $t \neq 0$, the centre of $H_{t,c}^{\mf{gl}(n)}$ equals $\C[\mr{eu}_{c}] \o \C Z(W)$,
where $\eu_{c} \in H_{t,c}$ is the Euler element.
\end{prop*}

\subsection*{(Quantum) Hamiltonian reduction}

The action of $\T$ on the generalized Calogero--Moser space $\Spec Z(H_{0,c})$ is Hamiltonian and we have seen that the centre of $H_{0,c}^{\T}$ is essentially given by the ring of invariants $Z(H_{0,c})^{\T}$. In applications to symplectic singularities, it is natural to consider instead the Hamiltonian reduction of $\Spec Z(H_{0,c})$ with respect to $\T$. This is done in Section~\ref{sec:QHRtorus}. For $\zeta \in \C = (\mr{Lie} \, \T)^*$, we define the algebra of quantum Hamiltonian reduction to be the quotient
\[
A_{t,c,\zeta} = A_{t,c,\zeta}(W) = \eo H_{t,c}^{\mf{gl}(n)}\eo / \langle \eo \mr{eu}_{c} - \zeta \rangle,
\]
and the Hamiltonian reduction of $\Spec Z(H_{0,c})$  is
\[
\mu^{-1}(\zeta) \git \, \T = \Spec \left( Z(H_{0,c})^{\T} / \langle \mr{eu}_{c} - \zeta \rangle \right),
\]
where $\mu$ is the moment map. Let $e=\frac{1}{|W|} \sum_{w\in W} w$. As a consequence of the fact that the double centralizer property holds for the $(A_{t,c,\zeta},eA_{t,c,\zeta}e)$-bimodule $A_{t,c,\zeta} e$,  (see Theorem~\ref{thm:doublecetralH0}), we show that:

\begin{cor*}(\ref{cor:centreHamredT})
The centre of $A_{0,c,\zeta}$ is isomorphic to $Z(H_{0,c})^{\T} / \langle \mr{eu}_{c} - \zeta \rangle$.     
\end{cor*}

Rational Cherednik algebras provide a means to study Poisson deformations of the symplectic quotient singularity $(\h \times \h^*)/W$ because the centre of $H_{0,0}$ can be identified with the ring of functions on this quotient. Let $\mc{O}_{\mr{min}} \subset \mf{gl}(n)$ be the minimal nilpotent orbit with respect to the adjoint action of $\GL(n)$. Corollary \ref{cor:centreHamredT} implies that the centre of $A_{0,0,0}$ can be identified with the ring of functions on the symplectic singularity $\overline{\mc{O}}_{\mr{min}}/W$ and hence $A_{t,c,\zeta}$ can be used to study Poisson deformations of this singularity. Our first result in this direction is the following. 

\begin{thm*}(\ref{thm:filteredquantofA0})
    For $t \neq 0$, the algebra $e A_{t,c,\zeta} e$ is a filtered quantization of $\overline{\mc{O}}_{\mr{min}}/W$. 
\end{thm*}

The algebra $A_{t,c,\zeta}$ can also be identified with a summand of the global sections of a sheaf of Cherednik algebras on the projective space $\mathbb{P}(\h)$, as studied in \cite{ProjAffine}. Since the latter sheaf of algebras has finite global dimension and localization holds for Weil generic parameters, we deduce that:

\begin{thm*}(\ref{thm:finiteglobaldimangular})
	For Weil generic $(t,c,\zeta)$, $A_{t,c,\zeta}$ has finite global dimension. 
\end{thm*}

At $t = 0$, we can consider the commutative algebras $Z(A_{0,c,\zeta})$ as forming a flat family. More precisely, let $\mf{\tilde c}$ be the parameter space of all possible pairs $(c,\zeta)$. Then there exists a $\C[\mf{\tilde c}]$-algebra $A_0$ such that
$A_{0, c, \zeta} = A_0 \otimes_{\C[\mf{\tilde c}]} \C$, where $\C[\mf{\tilde c}]$ acts on $\C$ by evaluation. More importantly, one has the flatness of the centre, that is 
$Z(A_{0, c, \zeta}) = Z(A_0) \otimes_{\C[\mf{\tilde c}]} \C$.


As one would hope, this gives rise to Poisson deformations of the symplectic singularity $\overline{\mc{O}}_{min}/W$. 


\begin{thm*}(\ref{sympsing})
	$\Spec Z(A_{0}) \to \mf{\tilde c}$ is a graded Poisson deformation of the symplectic singularity $\overline{\mc{O}}_{min}/W$.  
\end{thm*} 

This result should provide a representation-theoretic way to enumerate the projective $\Q$-factorial terminalizations of the singularity $\overline{\mc{O}}_{min}/W$.

 \subsection*{Outline of the article.}
 
Section~\ref{sec:gradedfilered} contains some basic results on graded and filtered algebras that we will use later. Then Section~\ref{sec:RCAdefn} recalls the definition of rational Cherednik algebras and outlines their basic properties. In Section~\ref{sec:twosubalgebras}, we begin the study of the two subalgebras. We give explicit presentations of these algebras and show that they are rings of invariants. The behaviour of $H_{t,c}^{\mf{so}(n)}$ is exceptional when $n = 2$ and we study this case separately in Section~\ref{sec:rank2centre}, giving an explicit description of the centre of the algebra (which does not depend on the parameters $t,c$). 
In Section~\ref{sec:invariantsub}, we study the properties of the invariant rings $H_{t,c}^{\Gamma}$. The case $t=0$ is considered in detail 
in Section~\ref{sec:invariantt=0}. 
 Finally, in Section~\ref{sec:QHRtorus} we study the quantum Hamiltonian reduction of $H_{t,c}$ with respect to the action of $\T$.

\subsection*{Conventions}

We work throughout over the complex numbers $\C$. Unadorned tensor product will mean $\o_{\C}$. 

\subsection*{Acknowledgements}

The first author was partially supported by a Research Project Grant from the Leverhulme Trust and by the Engineering and Physical Sciences Research Council [grant numbers EP/W013053/1, EP/R034826/1]. The second author was supported by the Engineering and Physical Sciences Research Council [grant number EP/W013053/1].

\section{Filtered and graded algebras}\label{sec:gradedfilered} 

In this section, we recall basic facts about filtered and graded algebras that will be required later. The centre of an algebra $A$ is denoted $Z(A)$. 

\subsection{Graded algebras}\label{sec:gradedalgebras}

For a $\Z$-graded vector space $M$, $M_i$ will denote the degree $i$ part. Let $B$ be a $\C$-algebra. A ($\Z$-)grading on $B$ is a direct sum decomposition $B = \bigoplus_{i \in \Z} B_i$ such that $B_i \cdot B_j \subset B_{i+j}$. The fact that $B$ is $\Z$-graded can be interpreted as an action of $\Cs$ on $B$ by algebra automorphisms such that 
\[
B_i := \{ b \in B \, | \, \lambda \cdot b = \lambda^i b, \ \forall \lambda \in \Cs \}. 
\]
Then $B^{\Cs} = B_0$.

\subsection{Filtered algebras}

Let $A$ be a $\C$-algebra. A filtration $(\mc{F}_i)_{i \in \Z}$ of $A$ is a nested collection of subspaces $\cdots \subset \mc{F}_i \subset \mc{F}_{i+1} \subset \cdots$ such that $\mc{F}_i \mc{F}_j \subset \mc{F}_{i+j}$. The \textit{associated graded algebra} of a filtered algebra $A$ is
\[
\gr A = \gr_{\mc{F}} A = \bigoplus_{i \in \Z} \, (\gr_{\mc{F}} A)_i, \quad \textrm{ with } \quad (\gr_{\mc{F}} A)_i = \mc{F}_i / \mc{F}_{i-1}.
\]
Given $a \in A$ we denoted by $\sigma(a)$ the corresponding element in $\gr A$, $\sigma(a)$ is called symbol of $a$.

The following summarizes basic results on filtered algebras that we need. 

\begin{lem}\label{lem:gradedin}
    Let $A$ be a $\C$-algebra equipped with a filtration $(\mc{F}_i)_{i \in \Z}$ such that $\mc{F}_{-1} = 0$ and $\bigcup_i \mc{F}_i = A$.
    \begin{enumerate}
        \item[(i)] Let $B \subset A$ be a subalgebra filtered by restriction. Then $\gr_{\mc{F}} B$ is a subalgebra of $\gr_{\mc{F}} A$ and $\gr_{\mc{F}} B = \gr_{\mc{F}} A$ implies that $B = A$.
        \item[(ii)] $\gr_{\mc{F}} Z(A) \subset Z(\gr_{\mc{F}} A)$. 
        \item[(iii)] Assume that there exists a reductive group $G$ acting on $A$ such that each $\mc{F}_i$ is $G$-stable. Then $G$ acts on $\gr_{\mc{F}} A$ and $\gr_{\mc{F}}(A^G) = (\gr_{\mc{F}} A)^G$. 
    \end{enumerate} 
\end{lem}

Recall that an element $u$ of an algebra $U$ is \textit{normal} if $u U = Uu$ is a two-sided ideal in $U$. Also, $u$ is said to be \textit{regular} if $ur = 0$ or $r u = 0$ implies $r = 0$, where $r \in U$. 

\begin{lem}\label{lem:assgrregularelement}
	Let $U$ be a filtered algebra and $u \in U$ a normal element such that its symbol $\sigma(u) \in \gr U$ is regular. Then $\gr (U / \langle u \rangle) = (\gr U) / \langle \sigma (u) \rangle$.  
\end{lem}

\begin{proof}
Let $I = \langle u \rangle$ and assume $u$ has degree $d$. By definition of the quotient filtration, the exact sequence $0 \to I \to U \to U / I \to 0$ is strictly filtered. This means that the associated graded sequence $0 \to \gr I \to \gr U \to \gr(U / I) \to 0$ is also exact. Therefore, it suffices to show that $\gr I = \langle \sigma(u) \rangle$. Since $u$ is normal, an element of $\gr I$ is of the form $\sigma(u r)$ for some $r \in U$. If $r \in \mc{F}_i U$ then 
 $$
 0 \neq \sigma(u) \sigma(r) = ur \, \mod \, \mc{F}_{i+d-1} \in \gr_{i+d} U. 
 $$
 This implies that $ur \in \mc{F}_{i+d} \setminus \mc{F}_{i+d-1}$ and hence $\sigma(ur) = \sigma(u) \sigma(r)$. In particular, $\sigma(ur) \in \langle \sigma(u) \rangle$. Hence $\gr I \subset \langle \sigma(u) \rangle$. The opposite inclusion is obvious.  
\end{proof}

\section{The rational Cherednik algebra}\label{sec:RCAdefn}

In this section we recall the definition of rational Cherednik algebras associated to complex reflection groups, following \cite{EG}, \cite{DunklOpdam}, \cite{GGOR}. 

Let $\mathfrak{h}$ be a complex vector space. Recall that a \textit{complex reflection} is an invertible linear map $s \colon \h \to \h$ of finite order that fixes a hyperplane $\Ker \alpha_s$ pointwise for some $\alpha_s \in \h^*$. That is, the rank of $(s-1)$ is one.
Let $\alpha_s^\vee\in \h$ be an eigenvector of $s$ with the eigenvalue different from 1. Let $\lambda_s\in \C$ be the non-trivial eigenvalue of $s$ on $\h^*$: $s(\alpha_s)=\lambda_s \alpha_s$.

\begin{defn}
A {\it complex reflection group} $W$ is a finite subgroup of $GL(\h)$ generated by complex reflections. The set of complex reflections in $W$ is denoted $\mc{S}$. 
\end{defn}
 We will also be denoting a complex reflection group as a pair $(W, \mathfrak{h})$. A complex reflection group $(W,\h)$ is said to be \textit{irreducible} if $\h$ is an irreducible representation of $W$. 

We assume throughout that $(W,\h)$ is irreducible. Note that $Z(W) = \Z / {\ell} \Z$ because $Z(W) \subset \Cs \Id_{\h}$, using the fact that $\h$ is irreducible. The order of the centre of an irreducible complex reflection group is the greatest common divisor of the degrees of a set of homogeneous fundamental invariants of $W$ \cite[Theorem~1.2]{CohenReflections}. In particular, for the infinite series $G(m,p,n)$ of irreducible complex reflection groups, $Z(G(m,p,n)) = \Z / (m/p) \Z$. 


\begin{defn}\label{ratcherdef}
Fix a complex reflection group $(W,\mathfrak{h})$, $t \in \C$ and a conjugation invariant function $c \colon \mc{S} \rightarrow \mathbb{C}$. The associated  \textit{rational Cherednik algebra} $H_{t,c}= H_{t,c}(W)$  is the quotient of the skew-group ring $T(\h \oplus \h^*) \rtimes W$ by the relations
\[
x\otimes x' - x' \otimes x, \ y\otimes y' - y' \otimes y, \  
y\otimes x-x\otimes y-\kappa(x,y), \quad \forall \, x, x' \,\in \mathfrak{h}^*, y, y' \, \in \mathfrak{h},
\]
where $T(\h \oplus \h^*)$ is the tensor algebra on $\h \oplus \h^*$ and 
\[
\kappa(x,y)=  t  x(y) - \sum_{s\in \mc{S}}\frac{2 c_s}{\alpha_s(\alpha_s^\vee)}x(\alpha_s^\vee) \alpha_s(y) s \in \C W.
\]
\end{defn}
The space of parameters for the rational Cherednik algebra is $\C \oplus \Hom_W(\mc{S},\C)$. For ease of notation, we write $H_{t,c}(W)$ as $H_{t,c}$ whenever $W$ is clear from the context.

\subsection{The Euler element}
Let $n=\dim \h$, and let us choose a basis $y_1,\ldots, y_n\in \h$, and a dual basis $x_1,\ldots, x_n \in \h^*$. 
The Euler element $\eu_{c} \in H_{t,c}(W)$ is defined to be
\[
\eu_{c} = \sum_{i=1}^n x_i y_i - \sum_{s \in \mc{S}} \frac{2 c_s}{1-\lambda_s}s. 
\]
It does not depend on a basis, and it satisfies
\begin{equation}
\label{eurelations}
[\eu_{c},x] = t x, \quad [\eu_{c},y] = -t y, \quad [\eu_{c},w] = 0, 
\end{equation}
for all $x \in \h^*, y \in \h$ and $w \in W$. 
Indeed, in order to check the first equality we consider
\begin{equation}
\label{euxk}     
[\eu_{c}, x] = \sum_{i=1}^n x_i [y_i, x] - \sum_{s\in\mc{S}}\frac{2 c_s}{1-\lambda_s}[s,x],
\end{equation}
where \begin{equation}
\label{Six}
[y_i, x] 
= t x(y_i) -\sum_{s\in \mc{S}} \frac{2 c_s}{\alpha_s(\alpha_s^\vee)} \alpha_s(y_i) x(\alpha_s^\vee) s.
\end{equation}
Then substitution of \eqref{Six} into \eqref{euxk} gives 
$$
[\eu_{c}, x] = t x - \sum_{s\in \mc{S}} \frac{2 c_s}{\alpha_s(\alpha_s^\vee)} x(\alpha_s^\vee) \alpha_s s
- \sum_{s\in\mc{S}}\frac{2 c_s}{1-\lambda_s}(s(x)-x)s.
$$
Let us choose $\alpha_s$ and $\alpha_s^\vee$ so that $s(x)=x-x(\alpha_s^\vee)\alpha_s$. Then  $\lambda_s = 1-\alpha_s(\alpha_s^\vee)$ and we get that $[\eu_{c},x]=0$. Similarly, 
$$
[\eu_{c}, y] = - t y +  \sum_{s\in \mc{S}} \frac{2 c_s}{\alpha_s(\alpha_s^\vee)} \alpha_s(y) s (\alpha_s^\vee) s 
- \sum_{s\in\mc{S}}\frac{2 c_s}{1-\lambda_s}(s(y)-y)s = -t y
$$
since $s(\alpha_s^\vee)= \lambda_s^{-1} \alpha_s^\vee$. 
The  remaining relation in \eqref{eurelations} can also be checked.

 The algebra $H_{t,c}(W)$ is $\Z$-graded with $\deg y= -1, \deg x = 1$ and $\deg w = 0$ for $x \in \h^*, y \in \h$ and $w \in W$. We see from the above equations that this grading is inner when $t =1$, given by the adjoint action of $\eu_{c}$. As explained in Section~\ref{sec:gradedalgebras}, this can be thought of as equipping with an action of $\Cs$. Since we wish to consider $\Cs$ as a maximal torus in $\SL_2$ (at least when $W$ is a real reflection group) we write $\T := \Cs$ from now on. Then $H_{t,c}(W)_0 = H_{t,c}(W)^{\T}$.

\subsection{Automorphisms}

We have explained that $\T$ acts on $H_{t,c}$ by algebra automorphisms. When $W$ is a real reflection group, there is an action of $\SL_2$ on $H_{t,c}$, realising $\T$ as the maximal torus of diagonal matrices. This action was first defined in \cite[Corollary~5.3]{EG} and we recall it here. Since $W$ is real, we can fix a $W$-equivariant isomorphism $\h \iso \h^*$ given by $v \mapsto v^*$. Then $\SL_2$ acts trivially on $W$ and
\begin{equation}\label{eq:SL2actioneq}
    g \cdot v = a v + b v^*, \quad g \cdot v^* = c v + d v^*, \quad \textrm{for } g = \left(\begin{array}{cc} 
a & b \\
c & d 
\end{array} \right) \in \SL_2, \textrm{ and } v \in \h.  
\end{equation}


\section{Two subalgebras: the presentations}\label{sec:twosubalgebras}

In this section we introduce two subalgebras of a rational Cherednik algebra $H_{t, c}(W)$, and we find their defining relations. We denote these subalgebras as $H_{t, c}^{\mathfrak{g}}(W)$, where $\mathfrak{g}=\mathfrak{gl}(n)$ or $\mathfrak{g}=\mathfrak{so}(n)$. In the latter case $W$ is assumed to be a real reflection group. The algebra $H_{t, c}^{\mathfrak{g}}(W)$ at $t=1$  is a flat deformation of the skew product of a quotient of the universal enveloping algebra $U({\mathfrak{g}})$ with the group $W$, which is established in \cite{Feiginangular} for the case of real reflection groups $W$.

\subsection{The degree zero subalgebra}
\label{Degree0}
Let $(W,\h)$ be an irreducible complex reflection group. Let $x_1, \ds, x_n$ be a basis of $\h^*$ and let $y_1, \ds, y_n \in \h$ be the dual basis.


\begin{defn}
\label{degree-zero-def}
     The 
    \textit{degree zero subalgebra} 
    $H_{t,c}^{\mf{gl}(n)} = H_{t,c}^{\mf{gl}(n)}(W)$ of $H_{t,c}(W)$ is the subalgebra generated by all $E_{ij} = x_i y_j$ and the group $W$. 
\end{defn}

Note that the $E_{ij}$ are a basis of $\h \o \h^*$ in $H_{t,c}$. Thus, $H_{t,c}^{\mf{gl}(n)}$ can equivalently be described as the subalgebra generated by $\h \o \h^*$ and $W$. 

The algebra $H_{t,c}$ admits a filtration $\mc{F}_i$ with $\h,\h^* \subset \mc{F}_1$ and $W \subset \mc{F}_0$. We give all subalgebras of $H_{t,c}$ a filtration by restriction, also denoted $\mc{F}_i$. In particular, $H_{t,c}^{\mf{gl}(n)}$ is a filtered algebra. Let $P = \C[ \h \times \h^* ]$. The terminology of Definition \ref{degree-zero-def} is justified by the following lemma.

\begin{lem}\label{glisevendegree}
    $H_{t,c}^{\mf{gl}(n)} = H_{t,c}^{\T}$ is the subalgebra of all elements of degree zero.
\end{lem}

\begin{proof}
    The inclusion $H_{t,c}^{\mf{gl}(n)} \subset H_{t,c}^{\T}$ is immediate since  $H_{t,c}^{\mf{gl}(n)}$ is generated by elements of degree zero. For the opposite inclusion, we equip both algebras with the subalgebra filtration coming from $H_{t,c}$. Then Lemma~\ref{lem:gradedin}(i) implies that it suffices to show equality of associated graded algebras. The filtration $\mc{F}_i$ on $H_{t,c}$ is $\T$-stable. Therefore Lemma~\ref{lem:gradedin}(iii) says that $\gr H_{t,c}^{\T} =  (\gr H_{t,c})^{\T}$  and it is a consequence of the PBW Theorem for rational Cherednik algebras \cite[Theorem~1.3]{EG} that
    \[
    (\gr H_{t,c})^{\T} = (P \rtimes W)^{\T} = P^{\T} \rtimes W.
    \]
    On the other hand, $\gr H_{t,c}^{\mf{gl}(n)}$ certainly contains the $E_{ij}$ and $W$. Therefore, it suffices to note that the $E_{ij} \in P$ generate $P^{\T}$ and hence $P^{\T} \rtimes W \subset \gr H_{t,c}^{\mf{gl}(n)}$. 
\end{proof}

\begin{cor}\label{cor:grglalg}
    $\gr H_{t,c}^{\mf{gl}(n)} = P^{\T} \rtimes W=(\gr H_{t,c})^{\T}$.
\end{cor}

Recall that
\begin{equation}
\label{Sjk}
S_{jk}
:=[y_j, x_k] 
= t \delta_{jk} -\sum_{s\in \mc{S}} \frac{2 c_s}{\alpha_s(\alpha_s^\vee)} \alpha_s(y_j) x_k(\alpha_s^\vee) s.
\end{equation}
Let us rescale elements $\alpha_s \in \h^*$, $\alpha_s^\vee\in \h$ so that a reflection $s\colon {\mathfrak h} \to \mathfrak h$ has the form $s(v)= v - \alpha_s(v) \alpha_s^\vee$ for any $v \in {\mathfrak h}$.

The next statement generalises a result from \cite{Feiginangular} to arbitrary $t$ and complex reflection groups. 

\begin{thm}
\label{glnpresentation}
    The algebra $H_{t, c}^{\mathfrak{gl}(n)}$ is isomorphic to the quotient of the skew product of the tensor algebra $T(\h^*\otimes \h)\rtimes {W}$ by the relations  
    \begin{align}
         [E_{ij}, E_{kl}]=
        t \delta_{jk} E_{il} - t \delta_{il}E_{kj}+
      \sum_{s \in \mc{S}} \frac{2 c_s}{\alpha_s(\alpha_s^\vee)} \bigl( \alpha_s(y_l) x_i(\alpha_s^\vee) E_{kj}
       -\alpha_s(y_j) x_k(\alpha_s^\vee) E_{il} \bigr)s 
          \nonumber
       \\
       \label{relG1}
     + \sum_{s \in \mc{S}} \frac{2 c_s}{\alpha_s(\alpha_s^\vee)}  \alpha_s(y_j)\alpha_s(y_l) 
     \sum_{r=1}^n x_r(\alpha_s^\vee)
     \bigl(x_k(\alpha_s^\vee) E_{ir}-
     x_i(\alpha_s^\vee) E_{kr}\bigr)s, 
          \\
          \label{relG2}
        E_{ij} E_{kl} - E_{i l} E_{k j} = 
        t \delta_{jk} E_{il}- t \delta_{lk}E_{ij}+
        \sum_{s \in \mc{S}} \frac{2 c_s}{\alpha_s(\alpha_s^\vee)} x_k(\alpha_s^\vee) \bigl(\alpha_s(y_l) E_{ij}- \alpha_s(y_j)E_{il}\bigr)s.
    \end{align}
A basis of the algebra $H_{t, c}^{\mathfrak{gl}(n)}$ is given by the elements
\begin{equation}
\label{basisgl}
E_{i_1 j_1}^{m_1}\ldots E_{i_k j_k}^{m_k}w,
\end{equation}
where $k\in \Z_{\ge 0}$, $m_s \in \N$, $w\in W$ and
\begin{equation}
\label{basisglrestrictions}    
1\le i_1 \le \ldots \le i_k \le n, \quad 1\le j_1 \le \ldots \le j_k \le n,  \text{ and } 
(i_s,i_{s+1}) \ne (j_s, j_{s+1})
\text{ for } 1\le s \le n-1.
\end{equation}
\end{thm}
\begin{proof}
For the elements $E_{ij} = x_i y_j \in H_{t, c}$ we have that 
$$
[E_{ij}, E_{kl}] = x_i S_{jk} y_l - x_k S_{li} y_j,
$$
where $S_{jk}=[y_j, x_k]$ has the form \eqref{Sjk}. This implies relation \eqref{relG1} in $H_{t, c}^{\mathfrak{ gl}_n} \subset H_{t, c}$. Similarly, relation \eqref{relG2} can be checked by making use of the equality 
$$
E_{ij} E_{kl} - E_{il} E_{kj} = x_i S_{jk} y_l - x_i S_{lk} y_j. 
$$
It is also clear that the algebra $H_{t, c}^{\mathfrak{gl}(n)}$ is generated by the elements $E_{ij}$ ($1\le i, j \le n$) and $w \in W$. Note that elements $\eqref{basisgl}$ with the specified restrictions \eqref{basisglrestrictions} span the algebra $H_{t, c}^{\mathfrak{gl}(n)}$. Indeed, by moving elements of the group to the right and by the relation \eqref{relG1} together with an induction on the degree in $E_{ij}$, any element of the algebra can be represented as a linear combination of monomials \eqref{basisgl}, where restriction $1 \le i_1 \le \ldots \le i_k \le n$ holds. Then by employing relations \eqref{relG2} and induction on the degree in $E_{ij}$ one can assume that all the restrictions \eqref{basisglrestrictions} hold. 

We are left to show that elements \eqref{basisgl} from the algebra $H_{t, c}^{\mathfrak {gl}_n}$  with the specified restrictions \eqref{basisglrestrictions} are linearly independent. Observe that any element 
\begin{equation}
    \label{monomialgl}
\prod_{i=1}^n x_i^{a_i} y_i^{b_i}w \in H_{t, c}^{\mathfrak {gl}_n}, \quad a_i, b_i \in \Z_{\ge 0}, \, w \in W   
\end{equation}
from the standard PBW basis of the algebra $H_{t, c}$ can 
be uniquely represented as a linear combination of elements \eqref{basisgl}, \eqref{basisglrestrictions}.
Indeed, there will be a unique term $h$ in the linear combination of degree $\sum_{k=1}^n a_k = \sum_{k=1}^n b_k$ in elements $E_{ij}$. If $a_1$ and $b_1$ are both positive then the first factor of $h$ has to be $E_{11}^{\min (a_1, b_1)}$. If $a_2$ and $b_2$ are both positive then the second factor will be $E_{21}^{\min(a_2, b_1-a_1)}$ or $E_{12}^{\min (a_1-b_1, b_2)}$ or $E_{22}^{\min(a_2, b_2)}$ in the cases $b_1>a_1$ or $a_1> b_1$ or $a_1=b_1$, respectively. This process can be continued to determine $h$ completely. In the cases when some $a_i$ or $b_i$ vanish, the corresponding variable $x_i$ or $y_i$ is not present in the monomial \eqref{monomialgl} and the index $i$ does not appear in the first or second place, respectively, in the factors of $h$, otherwise the process of building the element $h$ is similar to the above. 

This proves that elements \eqref{basisgl}, \eqref{basisglrestrictions} form a basis of the algebra, which also implies that relations \eqref{relG1}, \eqref{relG2} are complete. 
\end{proof}

\subsection{The Dunkl angular momentum subalgebra} 
Now let $W$ be a  finite real reflection group. Let $y_1, \ldots, y_n$ be an orthonormal basis of $\mathfrak h$ for a fixed $W$-invariant inner product $(\cdot, \cdot )$. Let $x_1, \ldots, x_n$ be the dual basis of ${\mathfrak h}^*$. Note that $\mathfrak{h}^*\cong \mathfrak h$ as $W$-modules. Define elements $M_{ij}= x_i y_j - x_j y_i \in H_{t, c}(W)$, where $1\le i, j \le n$. Note that $M_{ij}=-M_{ji}$. 

\begin{defn}(cf. \cite{Feiginangular})
     The 
    \textit{Dunkl angular momentum subalgebra} 
    $H_{t,c}^{\mf{so}(n)} = H_{t,c}^{\mf{so}(n)}(W)$ of $H_{t,c}(W)$ is the subalgebra generated by all $M_{ij}$ and the group $W$. 
\end{defn}

In the real case we have $\alpha_s(\alpha_s^\vee)=2$ and formula (\ref{Sjk}) takes the form
\begin{equation}
\label{Sjkreal}
S_{jk}=[y_j, x_k] 
= t \delta_{jk} -\sum_{s\in \mc{S}}  c_s \alpha_s(y_j) x_k(\alpha_s^\vee) s.
\end{equation}

Consider elements $m_{ij} \in \Lambda^2 \mathfrak h$ given by $m_{ij}= y_i\wedge y_j$. The next statement was established in \cite{Feiginangular} for $t=1$. The case of general $t$ is similar and we give details below.

\begin{thm}\label{sobasis}
    The algebra $H_{t, c}^{\mathfrak{so}(n)}$ is isomorphic to the quotient $B$ of the skew product of the tensor algebra $T(\Lambda^2 \mathfrak{h})\rtimes {W}$ by the relations  
    \begin{align}
    \label{relG1so}
         [m_{ij}, m_{kl}]=
    m_{il} S_{jk} + m_{jk} S_{il} - m_{ik} S_{j l} - m_{jl} S_{ik},
          \\
          \label{relG2so}
        m_{ij} m_{kl}+m_{jk} m_{il}+m_{ki} m_{jl} = m_{ij} S_{kl} + m_{jk} S_{il}+ m_{ki} S_{jl}.
   \end{align}
            The isomorphism is given by the map $\rho\colon B \to H_{t, c}^{\mathfrak{so}(n)}$ such that $\rho(m_{ij})=M_{ij}$ and $\rho(w)=w$ for $w \in W$.
A basis of the algebra $H_{t, c}^{\mathfrak{so}(n)}$ is given by the elements
\begin{equation}
\label{basisso}
M_{i_1 j_1}^{r_1}\ldots M_{i_k j_k}^{r_k}w,
\end{equation}
where $k\in \Z_{\ge 0}$, $r_s \in \N$, $w\in W$ and
\begin{equation}
\label{basissorestrictions}    
i_s < j_s, \quad
1\le i_1 \le \ldots \le i_k \le n, \quad  i_s=i_{s+1} \implies j_s < j_{s+1}, \quad  
i_s < i_{s'} < j_{s} \implies j_{s'}\le j_s.
\end{equation}
\end{thm}
\begin{proof}
Consider the isomorphism $\varphi\colon \mathfrak{h} \to \mathfrak{h}^*$ given by $\varphi(v)(u)=(v, u)$ for any $v, u \in \mathfrak h$. Under this isomorphism $\varphi(y_i)=x_i$ and $\varphi(\alpha_s^\vee)=\frac12 (\alpha_s^\vee, \alpha_s^\vee) \alpha_s$ since $\alpha_s(\alpha_s^\vee)=2$ and by orthogonality of the reflection $s\in \mc{S}$. We also have $\beta(v) = \varphi(v)(\varphi^{-1}(\beta))$ for any $v \in \mathfrak{h}, \beta \in \mathfrak{h}^*$. Therefore
\begin{equation}
\label{checkcancel}
    \alpha_s (y_j) x_i (\alpha_s^\vee) =
    x_j\left(\frac{2\alpha_s^\vee}{(\alpha_s^\vee, \alpha_s^\vee)}\right) \frac{(\alpha_s^\vee, \alpha_s^\vee)}{2} \alpha_s(y_i) 
    = \alpha_s(y_i) x_j(\alpha_s^\vee).
\end{equation}
Note that $M_{ij}=E_{ij}- E_{ji}$, and elements $E_{ij}$ satisfy relations given in Theorem \ref{glnpresentation}. We apply the first relation of Theorem \ref{glnpresentation} to
$$
[M_{ij}, M_{kl}] =[ E_{ij}, E_{kl}] - [E_{ji}, E_{kl}]-[E_{ij}, E_{lk}]+[E_{ji}, E_{lk}],
$$
and we simplify the result by making use of relation \eqref{checkcancel}. After collecting terms at each $M_{ij}$ and cancellations of the terms coming from the second sum in the right-hand side of equality \eqref{relG1} we get 
\begin{equation}
    \label{soncommutators}
 [M_{ij}, M_{kl}]=
    M_{il} S_{jk} + M_{jk} S_{il} - M_{ik} S_{j l} - M_{jl} S_{ik}.
\end{equation}

Consider now
\begin{multline}
\label{crossing}
    M_{ij} M_{kl}+M_{jk}M_{il}+M_{ki}M_{jl} = (E_{ij} E_{kl}-E_{kj}E_{il})-(E_{ij}E_{lk}-E_{ik}E_{lj})-(E_{ji}E_{kl}-E_{ki}E_{jl})
    \\
    +(E_{ji}E_{lk}- 
    E_{jk}E_{li})+(E_{jk}E_{il} - E_{ik}E_{jl})+(E_{kj}E_{li}-E_{ki}E_{lj})
    \\
    = -(x_i S_{jl} y_k - x_i S_{kl}y_j) +(x_j S_{il} y_k - x_j S_{kl} y_i) +(x_k S_{jl} y_i - x_k S_{il} y_j),
\end{multline}
where we used the relation $S_{ij}=S_{ji}$ which follows from \eqref{checkcancel}. By the Jacobi identity,  
$[S_{jl}, y_k] = [S_{kl}, y_j]$. This implies that relation \eqref{crossing} can be rearranged to the form 
\begin{equation}
\label{crossingtrue}
   M_{ij} M_{kl}+M_{jk}M_{il}+M_{ki}M_{jl} = 
M_{ij}S_{kl}+M_{jk}S_{il}+M_{ki}S_{jl}.     
\end{equation}

Consider the map $\rho_0\colon \Lambda^2 \mathfrak h \to \langle M_{ij}| 1\le i < j \le n\rangle$ given by $\rho(m_{ij})=M_{ij}$, where in the codomain we have the linear span. The map $\rho_0$ is $W$-equivariant. It follows from relations \eqref{soncommutators} and \eqref{crossingtrue} that $\rho_0$ extends to a well-defined homomorphism of algebras $\rho\colon B \to H_{t, c}^{\mathfrak{so}(n)}$ by defining $\rho (w)=w$. It is also clear that $\rho$ is an epimorphism. Furthermore, monomials \eqref{basisso}  with restrictions \eqref{basissorestrictions} linearly generate the algebra $H_{t, c}^{\mathfrak{so}(n)}(W)$. This follows from a noncrossing visualisation of these monomials and since application of the relation \eqref{crossingtrue} allows to reduce the number of crossings in the monomials from the algebra $H_{t, c}^{\mathfrak{so}(n)}(W)$ (see \cite{Feiginangular}).
Consider now the highest symbol of the monomials \eqref{basisso} in the Dunkl embedding, it has the form
$L_{i_1 j_1}^{r_1}\ldots L_{i_k j_k}^{r_k}w$, where $L_{ij} = x_i p_j - x_j p_i\in \C[x_1, \ldots, x_n, p_1, \ldots, p_n]$. These highest symbols are linearly independent (see e.g. \cite{feigin2022algebra}), which implies that elements \eqref{basisso} with restrictions \eqref{basissorestrictions} form a basis of the algebra $H_{t, c}^{\mathfrak{so}(n)}(W)$. 

Similarly, by applying relations \eqref{relG1so}, \eqref{relG2so} we deduce that  the linear span of the corresponding monomials $m_{i_1 j_1}^{r_1}\ldots m_{i_k j_k}^{r_k}w \in B$ gives the entire algebra $B$. These monomials are linearly independent since $\rho$ maps them to a basis of $H_{t, c}^{\mathfrak{so}(n)}(W)$, and it follows that
 $\rho$ is an isomorphism.
\end{proof}

As a corollary we get the following statement.

\begin{prop}\label{gradedso}
    There is an isomorhism of algebras
\[
H_{0,0}^{\mathfrak{so}(n)} \cong \mathrm{gr}\,H_{t, c}^{\mathfrak{so}(n)}.
\]
\end{prop}

\begin{proof}
The isomorphism $H_{0,0} \cong \mathrm{gr}\, H_{t, c}$ restricts to an injective morphism
\[
H_{0,0}^{\mathfrak{so}(n)}\hookrightarrow \mathrm{gr}\, H_{t,c}.
\]
It is sufficient to show that the image of this map equals $\mathrm{gr}\,H_{t,c}^{\mathfrak{so}(n)}$. 

Theorem~\ref{sobasis} says that a basis of $H_{t,c}^{\mathfrak{so}(n)}$ is  given by the monomials \eqref{basisso} satisfying restrictions \eqref{basissorestrictions}. Consider a basis element $b= M_{i_1j_1}^{r_1}\cdots M_{i_kj_k}^{r_k}w$ and let $r= 2 \sum_{i=1}^k r_i$. Under the standard filtration $b\in \mc{F}_{r}$ and $b \not\in \mc{F}_{r-1}$. Then the symbol of $b$ in $\mathrm{gr}\, H_{t,c} = H_{0,0}$ is also $M_{i_1j_1}^{r_1}\cdots M_{i_kj_k}^{r_k}w$. Since Theorem \ref{sobasis} again says that all such monomials are a basis of $H_{0,0}^{\mathfrak{so}(n)}$ the equality $H_{0,0}^{\mathfrak{so}(n)} = \gr H_{t,c}^{\mathfrak{so}(n)}$ holds. 



\end{proof}

Let $\mf c$ be the parameter space $\C \oplus \Hom_W(\mc{S},\C)$. Then there exists a $\C[\mf{c}]$-algebra $H^{so(n)}$ which specialises to $H_{t,c}^{so(n)}$. 
An immediate consequence of Proposition~\ref{gradedso} is that:

\begin{cor}
The algebra $H^{so(n)}$ is flat over $\C[\mf{c}]$.
\end{cor}

The next result is key to understanding the algebra $H_{t,c}^{\mf{so}(n)}$ since rings of invariants are often much easier to study than algebras given by generators and relations. 

\begin{thm}\label{sonissl2}
There is equality $H_{t,c}^{\mf{so}(n)} =H_{t,c}^{\SL_2}$ of subalgebras of $H_{t,c}$. 
\end{thm}
\begin{proof}
First we show that there is an inclusion $H_{t,c}^{\mf{so}(n)}\subset  H_{t,c}^{\SL_2}$. 

As in \eqref{eq:SL2actioneq}, the action of any $g\in \SL_2$ is $g(x_i)=ax_i+b y_i$, $g(y_j)=c x_j+dy_j$, where $ad-bc=1$. The group $\SL_2$ acts trivially on $W$ by definition. It suffices then to show that $g$ fixes any generator $x_iy_j-x_jy_i$. We perform the calculation explicitly
\begin{equation}\label{eq:gSLaction}
    g(x_iy_j-x_jy_i)=adx_iy_j+bcy_ix_j-adx_jy_i-bcy_jx_i.
\end{equation}
Since $W$ is a real reflection group, \eqref{checkcancel} implies that $[x_i,y_j]=[x_j,y_i]$ and thus $y_ix_j-y_jx_i=x_jy_i-x_iy_j$. Therefore, \eqref{eq:gSLaction} can be rearranged to give
\[
adx_iy_j+bcx_jy_i-adx_jy_i-bcx_iy_j=(ad-bc)(x_iy_j-x_jy_i)=x_iy_j-x_jy_i.
\]
We conclude that $H_{t,c}^{\mf{so}(n)}\subset  H_{t,c}^{\SL_2}$. 

Identifying $\h \oplus \h^*$ with the space of $2 \times n$ matrices, it is a classical result in invariant theory \cite[Proposition 9.2]{FultonYOungTableaux} that the subalgebra $A$ of $P$ generated by the determinants $x_i y_j - x_j y_i$ is the invariant ring $P^{\SL_2}$. Since $H_0^{\mf{so}(n)} = A \rtimes W$ and $H_{0}^{\SL_2} = P^{\SL_2} \rtimes W$, we deduce that $H_{0}^{\mf{so}(n)} =H_{0}^{\SL_2}$. 

We have shown that $H_{t,c}^{\mf{so}(n)}\subset  H_{t,c}^{\SL_2}$. If we can show that $\mathrm{gr}\, H_{t,c}^{\mf{so}(n)}=\mathrm{gr}\, H_{t,c}^{\SL_2}$ the result will follow by Lemma~\ref{lem:gradedin}(i). As $\SL_2$ is a reductive group we can apply Lemma~\ref{lem:gradedin}(iii) to get $\mathrm{gr}\, H_{t,c}^{\SL_2}=(\mathrm{gr}\, H_{t,c})^{\SL_2}=H_0^{\SL_2}$. Combining this with Proposition~\ref{gradedso}, which says that $H_{0}^{\mathfrak{so}(n)} = \mathrm{gr}\,H_{t, c}^{\mathfrak{so}(n)}$, implies that $\mathrm{gr}\, H_{t,c}^{\mf{so}(n)}= H_0^{\mf{so}(n)}=H_0^{\SL_2}=\mathrm{gr}\, H_{t,c}^{\SL_2}$.
\end{proof}

\begin{remark}
    If $W$ is not a real reflection group then it does not preserve the linear span of the $M_{ij}$ and the algebra generated by the $M_{ij}$ and $W$ is larger than one would expect; in general this algebra will equal $H_{t,c}^{\mf{gl}(n)}$.
\end{remark}

\begin{remark}
    One can generalise the construction of the degree zero and the Dunkl angular momentum subalgebras as follows. Let $\mf{g} \subset \End_{\C}(\h) \cong \h \o \h^*$ be a reductive Lie subalgebra. Pick $W \subset N_{GL(\h)}(\mf{g})$ a (finite) reflection group. Then one can consider the subalgebra $H_{t,c}^{\mf{g}}$ of $H_{t,c}$ generated by $\mf{g}$ and $W$. It would be interesting to classify all pairs $(\mf{g},W)$ such that $H_{t,c}^{\mf{g}}$ defines a flat family of quotients of $U(\mf{g})\rtimes W$ when $t=1$.  
\end{remark}

\subsection{Subgroups of $\SL_2$} 

One can consider $H_{t,c}^{\Gamma}$ for other reductive subgroups $\Gamma \subset \SL_2$ and ask if there exists $\mf{g} \subset \h \o \h^*$ such that $H_{t,c}^{\Gamma} = H_{t,c}^{\mf{g}}$. We do not consider finite $\Gamma$ here. Other than $\T$ and $\SL_2$, this leaves the normalizer $N$ of $\T$ in $\SL_2$. 

Let $F$ be the automorphism of $H_{t,c}$ defined by $F(x_i) = y_i, F(y_i) = -x_i$ and $F(w) = w$. Then $F^2 = -1 \in \T$ and we can consider the action of the group $N = \langle \T, F \rangle$ on $H_{t,c}$. 

Recall that $M_{ij} := x_i y_j - x_j y_i$ and $H^{\mf{so}(n)}_{t,c}$ is generated by the $M_{ij}$ and $W$. 

\begin{lem}\label{include}
    There is a proper inclusion $H^{\mf{so}(n)}_{t,c} \subset H^{N}_{t,c}$. 
\end{lem}

\begin{proof}
    It is easy to check that $H_{t, c}^{\mf{so}(n)} \subset H_{t, c}^{N}$. Since this is an inclusion of filtered algebras, it suffices 
    to show that the inclusion is proper for the associated graded algebras. By Lemma \ref{lem:gradedin}(iii) and Theorem \ref{sonissl2} it is sufficient to establish that $H_{0, 0}^{\mf{so}(n)} \ne H_{0, 0}^{N}$.
    We have 
    $$
    H^{N}_{0,0} = (P \rtimes W)^{N} = P^{N} \rtimes W. 
    $$
    Then it suffices to show that the subalgebra $B$ of $P$ generated by the $M_{ij}$ is a proper subalgebra of $P^{N}$. Consider the elements $C_{ij,kl} := x_i x_j y_k y_l + x_k x_l y_i y_j$. They belong to $P^{N}$. We claim that they do not belong to $B$. For this, we define $\Delta \colon P^{N} \to \C[x,y]^{N} = \C[x^2 y^2]$ by $\Delta(x_i) = x, \Delta(y_i) = y$. Notice that $\Delta(M_{ij})= 0$ but $\Delta(C_{ij,kl}) = 2 x^2 y^2 \neq 0$. The claim follows. 
\end{proof}


Consider now the degree two component of $H_{0,0}^{N} = P^N \rtimes W$. Notice that $P_2^{\T} = \h \o \h^* \supset P_2^N$ since $\T \subset N$. Under the identification $\h \o \h^* = \End(\h)$, the automorphism $F$ corresponds to the map $A \mapsto - A^T$. Therefore, $P_2^N = P_2^{\SL_2} = \mf{so}(n)$ and hence $(\gr H_{t,c}^N)_2 = \mf{so}(n) \o \C W$. This, combined with Lemma~\ref{include}, implies that there is no subalgebra $\mf{g}$ of $\mf{gl}(n)$ such that $H^{N}_{t,c} = H^{\mf{g}}_{t,c}$. 

\subsection{The centralizer of $\mf{so}(n)$ in $W$}

Assume now that $(W,\mathfrak{h})$ is an arbitrary complex reflection group. Each element $m\in W$ acts as a linear transformation of $\mathfrak{h}$, denote this corresponding map $M\in \GL(\mathfrak{h})$. Recall that $W$ has a natural action on the dual space $\mathfrak{h}^*$, given by $(m\cdot y)(x)=y(m^{-1}x)$ and therefore the element $m\in W$ corresponds to $(M^{-1})^{\T}\in \GL(\mathfrak{h}^*)$. 

We can identify the space $\mathfrak{h}\otimes \mathfrak{h}^*$ with $\mathrm{End}(\mathfrak{h})$ via the map $\varphi\colon 
\mathfrak{h}\otimes \mathfrak{h}^* \to \mathrm{End}(\mathfrak{h})$ 
given by $\varphi(
x_i\otimes y_j)= x_iy_j$.  Equip $\mathrm{End}(\mathfrak{h})$ with a $W$-action by letting $w\cdot M=wMw^{-1}$. With this action the   map $\varphi$ becomes an isomorphism of $W$-modules. 
\begin{prop}\label{nomatrices}
Let $m\in W$ commute with all elements  $x_iy_j-x_jy_i$. If $n\geq 3$ then $m$ acts as a scalar multiple of the identity. If $n=2$ then $m$ acts as the sum of a scalar multiple of the identity and a skew-symmetric matrix.
\end{prop}

\begin{proof}
Using the identification $\varphi$ with $\mathrm{End}(\mathfrak{h})$, the pair $x_iy_j-x_jy_i$ corresponds to the matrix with $1$ in the $i^{th}$ row and $j^{th}$ column, $-1$ in the $j^{th}$ row and $i^{th}$ column and $0$ everywhere else. Let us denote this matrix by $P_{ij}$. Then for $m$ to commute with $x_iy_j-x_jy_i$ we require $m\cdot P_{ij}=MP_{ij}M^{-1}=P_{ij}$. Let $m_{ij}$ denote the entry in the $i^{th}$ row and $j^{th}$ column of $M$. Calculating $MP_{ij}M^{-1}=P_{ij}$ gives $m_{ij}=-m_{ji}$, $m_{ii}=m_{jj}$ and other entries of $M$ are $0$. There are two cases to consider.

If $\dim\, \mathfrak{h}\geq 3$ we must have that $m_{k\ell}=0$ for all $k\neq \ell$. This is because the commutation relation for $x_{1}y_2-x_2y_1$ implies that $m_{k\ell}=0$ for all $(k,\ell)\neq (1,2)$ or $(2,1)$. But then the commutation relation for $x_1y_3-x_3y_1$ forces $m_{k\ell}=0$ for all $(k,\ell)\neq (1,3)$ or $(3,1)$. The only non-zero terms of $M$ are the diagonal entries which are all equal.

If $\dim\, \mathfrak{h}=2$ then we need only consider the element $x_1y_2-x_2y_1$ and we have $m_{11}=m_{22}$ and $m_{12}=-m_{21}$. Hence the matrix associated to the action of $m$ has the form
\[
M=\begin{bmatrix}
a & b \\
    -b &  a
\end{bmatrix}
\]
for $a,\,b\in\C$, as required.
\end{proof} 

\begin{cor}\label{dim2}
    Suppose that $W\subset \mr{O}(n)$. Then the centraliser $Z_W(\mf{so}(n))$ of $\mf{so}(n)$ in $W$ equals $Z(W)$ for $n\ge 3$, and $Z_W(\mf{so}(2))=W \cap \mr{SO}(2)$. 
\end{cor}

If $W$ is a real reflection group then we identify $\mf{so}(n) \subset \End(\h) = \h \o \h^*$ with $\wedge^2 \h$ using the $W$-equivariant isomorphism $\h \iso \h^*$. 

\begin{lem}
    Assume that $W$ is an irreducible real reflection group and $\dim \h \ge 3$. Then the kernel of $W \to GL(\wedge^2 \h)$ equals $Z(W)$. 
\end{lem}

\begin{proof}
    Since $Z(W)$ is either trivial, or equals $\{ \pm 1 \}$, the kernel clearly contains $Z(W)$. Let $g \in W$ act trivially on $\wedge^2 \h$. Since $g$ has finite order, it is diagonalizable. Let  $\omega_1, \ds, \omega_n$ be  the eigenvalues of $g$ on $\h$. Recall that each $\omega_i$ is a root of unity. Then in a suitable basis $g$ acts on $\wedge^2 \h$ as a diagonal matrix with diagonal entries $\omega_1 \omega_2, \ds, \omega_{n-1} \omega_n$. If this is the identity matrix then $\omega_i \omega_j = 1$ for $i \neq j$. If $\omega_1 \notin \{ \pm 1 \}$ say then $\omega_2 = \omega_3 = \omega_1^{-1}$ and $\omega_3 = \omega_2^{-1}$ is a contradiction. Thus, $\omega_i \in \{ \pm 1 \}$ for all $i$. If $\omega_i = -1$ for some $i$ then this forces $\omega_i = -1$ for all $i$. Thus, $g = -\Id$ in this case. 
\end{proof}

 \section{The centre of the angular momentum algebra in rank two}\label{sec:rank2centre}

In this section we assume $\dim \h = 2$. Then $W = D_{\ell}$ is the dihedral group of order $2\ell$, for $\ell \ge 3$, and we write $D_{\ell} = \{ 1, \sigma, \ds, \sigma^{\ell-1}, s_1, \ds, s_{\ell} \}$, where $\sigma$ is a rotation and the $s_i$ are the reflections. Note that $s_i \sigma s_i = \sigma^{-1}$. Let $a_0, \ds, a_{\ell-1} \in \C \langle \sigma \rangle$ be the basis of the group algebra consisting of primitive idempotents. 
Explicitly, we have $a_i=\frac{1}{\ell} \sum_{k=0}^{\ell-1} \zeta^{- k i} \sigma^k$, where $\zeta$ is a fixed primitive $\ell$th root of unity.
Then $\sigma a_i = \zeta^i a_i$. 
In $\C D_{\ell}$ we have $s_i a_j = a_{-j} s_i$ where we define $a_{i + \ell} := a_i$ for $i \in \Z$.  We set 
\[
b_i = \left\{ \begin{array}{ll} 
a_i + a_{-i} & 2 i \neq 0 \, \mod \, \ell \\
a_i & 2i = 0 \, \mod \, \ell 
\end{array}
\right. \quad \textrm{for $i = 0,\ds, N := \lfloor \ell/2 \rfloor $}.
\]
We also set $d_i = a_i - a_{-i}$
for $i = 1, \ds, \widetilde N:=\lfloor \frac{\ell-1}{2} \rfloor$.
Note that $N+\widetilde N = \ell-1$.
Moreover, $d_i^2 = b_i$ and $d_i b_i = d_i$. 

\begin{lem}
    The spaces 
    $$
    Z_+ = \{ z \in \C \langle \sigma \rangle \, | \, s_i z = z s_i \, \forall \, i \}, \quad Z_- = \{ z \in \C \langle \sigma \rangle \, | \, s_i z = -z s_i \, \forall \, i \},
    $$
    have basis $b_0, \ds, b_N$ and $d_1,\ds, d_{\widetilde N}$ respectively.  
\end{lem}

\begin{proof}
Since $\langle \sigma \rangle$ is a normal subgroup of $D_{\ell}$, the latter acts by conjugation on $\C \langle \sigma \rangle$ and this action factors through the quotient $D_{\ell} / \langle \sigma \rangle = \langle s_1 \rangle$. The $a_i$ are a basis of $\C \langle \sigma \rangle$. Therefore, we get a basis of $Z_+$ and $Z_-$ respectively by applying the averaging operator $(1/2)(1 + s_1)$ (and rescaling if necessary), respectively the operator $(1/2)(1 - s_1)$. The result follows.      
\end{proof}

\begin{prop}\label{dihed}
    The centre of $H_{t,c}^{\mf{so}(2)}(D_{\ell})$ is a direct sum of $N+1$ polynomial rings. More specifically, when $\ell$ is even, 
    \[
    Z(H_{t,c}^{\mf{so}(2)}(D_{\ell})) = b_0 \C[M_{12}^2] \oplus \left( \bigoplus_{i = 1}^{N-1} b_i \C[ M_{12} d_i] \right) \oplus b_N \C[M_{12}^2],
    \]
    and when $\ell$ is odd, 
    \[
    Z(H_{t,c}^{\mf{so}(2)}(D_{\ell})) = b_0 \C[M_{12}^2] \oplus \left( \bigoplus_{i = 1}^{N} b_i \C[ M_{12} d_i] \right).
    \]
\end{prop}

\begin{proof}
The subalgebra $H_{t,c}^{\mf{so}(2)}(D_{\ell}) \subset H_{t,c}(D_{\ell})$ is generated by the element $M_{12}$ and elements $w \in D_{\ell}$. It has a basis $\{ M_{12}^j w \}$, where $j\in \Z_{\ge 0}$. The subalgebra is graded with $\deg M_{12} = 2$ and $\deg w = 0$. Therefore, its centre is also graded.  
We consider each graded piece in turn. In even degrees, we are looking for elements $M_{12}^{2k} z$, $k\in \Z_{\ge 0}$,  with $z \in \C D_{\ell}$ acting trivially on $M_{12}$ and commuting with $D_{\ell}$. By Corollary~\ref{dim2}, the first condition forces $z \in \C \langle \sigma \rangle$ since $\dim \h = 2$. Then, the second condition in turn means that $z$ belongs to $Z_+$. Hence, $z$ is a linear combination of the $b_i$ (which are all in the centre). In odd degrees, we have elements of the form $M_{12}^{2k+1} z$, again for some $z \in \C D_{\ell}$. For this to commute with $M_{12}$, we must have $[M_{12},z] = 0$ forcing $z \in \C \langle \sigma \rangle$. To commute with $D_{\ell}$ it suffices to check that $[s_1,M_{12}^{2k + 1} z] = 0$. This forces $s_{1} z = - z s_{1}$; that is, $z \in Z_-$. Thus, the element is a linear combination of the elements $M_{12}^{2k+1} d_i$. Since $d_i d_j = \delta_{i,j} b_i$ the result follows.     
\end{proof}

We note that for $t\ne 0$ the centre of the algebra $H_{t, c}^{\mf{so}(n)}$ is described in \cite{Feiginangular}, \cite{feigin2022algebra} for $n\ge 3$. However, it is missed in these references that the case $n=2$ requires different considerations as above.

\section{Invariant subalgebras}\label{sec:invariantsub}

In order to give a uniform treatment of the rings of invariants under both $\T$ and $\SL_2$, we make a number of general assumptions, which we check hold for both of these groups. 

\subsection{Setup}\label{sec:setupinvariant} Recall that the rational Cherednik algebra can be considered as an algebra $H$ over the polynomial ring $R=\C[\mathfrak{c}]$ such that specializing the parameters to some $(t,c) \in \mathfrak{c}$ gives $H_{t,c}$. We will consider automorphisms of $H$ that are $\C[\mathfrak{c}]$-linear and hence induce automorphisms of each $H_{t,c}$. 

Recall, e.g. \cite[\S~6.46]{MilneAlgebraicGroups}, that a connected linear algebraic group $\Gamma$ is said to be \textit{reductive} if its unipotent radical $R_u(\Gamma)$ is trivial. By Weyl's complete reducibility theorem, this is equivalent to either (a) there exists a finite-dimensional faithful semi-simple representation of $\Gamma$; or (b) every finite-dimensional representation of $\Gamma$ is semi-simple; see \cite[Theorem~22.42]{MilneAlgebraicGroups}.

Throughout, we fix $\Gamma$ to be a group of automorphisms of $H$ satisfying 
\begin{enumerate}
    \item[(I)] $\Gamma$ is reductive; and 
    \item[(II)] $\Gamma$ acts trivially on $W$ and $R$ and linearly on $\h \oplus \h^*$, preserving the symplectic form. 
\end{enumerate}
The group $W$ acts  by conjugation as inner automorphisms on $H$. The action of $Z(W)$ is trivial on $W$ and $R$ and linear on $\h \oplus \h^*$. Therefore, we may also impose that
\begin{enumerate}
    \item[(III)] the image of $Z(W) \to \mr{Aut}(H)$ lies in $\Gamma$. 
\end{enumerate}

In view of (II) we may write $H^{\Gamma \times W}=(H^W)^{\Gamma}=(H^{\Gamma})^W$ without ambiguity. The following elementary observation will be used repeatedly. 

\begin{lem}
    The group $\Gamma$ preserves the filtration on $H_{t,c}$ and hence acts as graded automorphisms of $H_{0,0} = P \rtimes W$. 
\end{lem}

\begin{proof}
Notice that since $\mc{F}_0 H_{t,c} = \C W$, $\mc{F}_1 H_{t,c} = \C W \oplus \h \oplus \h^*$ and 
\[
\mc{F}_k H_{t,c} = (\mc{F}_{k-1} H_{t,c})(\mc{F}_1 H_{t,c}(W))
\]
the claim follows from (II). 
\end{proof}


\subsection{GIT quotients}
We consider a representation $V$ of the group $\Gamma$. For $\Gamma= \T$ we assume $V = \h \oplus \h^*$, and for $\Gamma=\SL_2$ we assume $V=\h \o \C^2$. Let $\pi\colon V \to V\git \, \Gamma$ be the quotient map.

Recall that a representation $U$ of an algebraic group $\Gamma$ is said to be \textit{stable} if a generic $\Gamma$-orbit is closed in $U$. 

\begin{lem}\label{lem:stablerep}
 The representation $V$ is a stable representation.    \end{lem}

\begin{proof}
For $\Gamma=\T$, a generic orbit in $V$ is defined by the system of equations $x_i y_j = \alpha_{ij}$ for $1 \le i,j \le n$ and $\alpha_{ij} \in \C$ such that $\alpha_{ij} \neq 0$ for some pair $i,j$. In particular, it is closed and hence $V$ is stable. 

For $\Gamma=\SL_2$, we identify $V$ with the space of $2 \times n$ matrices, with $\SL_2$ acting by left multiplication; here $n = \dim \h$. Then $\C[V]^{\Gamma}$ is generated by the column determinants $\mr{det}_{i,j}$ for $1 \le i < j \le n$. Let 
us choose any $p \in V \git \, \Gamma$ other than $0$. Then there exists some $i < j$ such that $\mr{det}_{i,j}(A) \neq 0$ for all $A \in \pi^{-1}(p)$. Since $\SL_2$ acts freely (by left multiplication) on the space of all $2 \times 2$ matrices of fixed non-zero determinant, it implies that $\SL_2$ acts freely on $\pi^{-1}(p)$. Each fiber $\pi^{-1}(p)$ contains a unique closed $\SL_2$-orbit; see e.g. \cite[II~3.3~Satz~3 (b)]{Kraft}. It is the orbit of minimal dimension in the fiber. If $\pi^{-1}(p)$ was not a single $\SL_2$-orbit, then the closed $\SL_2$-orbit in $\pi^{-1}(p)$ would have dimension less than $\dim \SL_2$. In particular, it is not a free orbit. We deduce that $\pi^{-1}(p)$ is a single orbit and hence $V$ is stable. 
\end{proof}

\begin{prop}
    \label{prop:invariantringGorenstein}
        The spaces $V \git\, \Gamma$ and $V \git\, (\Gamma \times W)$ are Gorenstein.  
\end{prop}

The fact that the quotient $V \git\, \T$ is Gorenstein follows from standard results e.g. \cite[Corollary~6.4.3]{CohMac}. However, there is a general result of Knop, ensuring the quotient is Gorenstein \cite{KnopCanonical}, that can be 
applied in both cases which we do in the proof.

\begin{proof}
We consider first $V \git\, \Gamma$. 
The main result of \cite{KnopCanonical} states that $V \git \, \Gamma$ is Gorenstein if the following three conditions hold: 
    \begin{enumerate}
        \item[(a)] $V$ is a stable representation. 
        \item[(b)] $V_s := \{ v \in V \, | \, \dim \Gamma_v > 0 \}$ has codimension at least two in $V$. 
        \item[(c)] $\lambda_V = \lambda_{\mr{ad}}$. 
    \end{enumerate}
    Here, for a representation $U$, $\lambda_U$ is the linear character of $\Gamma$ given by the composed map $\Gamma \to \GL(U) \stackrel{\det}{\longrightarrow} \Cs$ and $\mr{ad}$ denotes the adjoint representation of $\Gamma$. We note that for both $\T$ and $\SL_2$, $\lambda_V = \lambda_{\mr{ad}}$ is the trivial character. Therefore (c) holds. Lemma~\ref{lem:stablerep} says that (a) holds. Thus, we only need to check (b).
    
    For $T$, the space $V_s$ equals $\{ 0 \}$. Therefore, (b) is satisfied. 

    For $\Gamma = \SL_2$, we make the identification of $V$ with the space of $2 \times n$ matrices, as in the proof of Lemma~\ref{lem:stablerep}. It follows from what is written there that $V_s \subset \pi^{-1}(0)$. In fact, we claim $V_s = \pi^{-1}(0)$. Notice that $\pi^{-1}(0)$ is precisely the set of $2 \times n$ matrices of rank at most $1$. Let $A$ have rank one. Then $g A = A$ for $g\in \SL_2$ if and only if $g$ acts trivially on $\mr{Im} \, A$. The subgroup of $\SL_2$ fixing a line pointwise is one-dimensional. Hence $A \in V_s$. The space of $2 \times n$ matrices of rank at most $1$ is $n+1$-dimensional. Therefore, we see that (b) holds if $n \ge 3$. When $n = 1$, $V \git \, \Gamma = \{ 0 \}$ and when $n = 2$, $V \git \, \Gamma = \mathbb{A}^1$ since $\C[V]^{\Gamma} = \C[\mr{det}_{1,2}]$. In both cases, $V \git \, \Gamma$ is Gorenstein.

    Now the argument for $V \git\, (\Gamma \times W)$ is similar. Since $W$ is finite, properties (a) and (b) follow immediately from the fact that they hold for $V$ as a $\Gamma$-module. For (c), $\lambda_V |_W$ and $\lambda_{\mr{ad}} |_W$ are both the trivial character (the former because $W \subset \mr{Sp}(V)$) so once again (c) follows from the fact that it holds for $V$ as a $\Gamma$-module. 
\end{proof}

\begin{lem}\label{lem:Wbaropenfree}
  Assume $\dim \h > 1$ for $\Gamma = \T$ and $\dim \h > 2$ for $\Gamma = \SL_2$. The action of  $\overline{W} = W / Z(W)$ on $V \git \, \Gamma$ is generically free. 
\end{lem}

\begin{proof}
Let $\Gamma = \T$ and $w \in W \setminus Z(W)$. Then we can choose $v \in \h$ and $\lambda \in \h^*$ such that neither is an eigenvector for $w$. Since $v,\lambda \neq 0$, the $\T$-orbit $\mc{O}_{(v,\lambda)}$ of $(v,\lambda)$ is closed but free. Thus, to show that $w \cdot \pi(v,\lambda) \neq \pi(v,\lambda)$, it suffices to show that $w(\mc{O}_{(v,\lambda)} ) \neq \mc{O}_{(v,\lambda)}$. But this follows from the fact that $(w(v),w(\lambda))$ does not belong to the line spanned by $(v,\lambda)$. 


Now let $\Gamma=\SL_2$. Consider a generic element $p\in V\git \, \Gamma$. The orbit passing through the preimage $\pi^{-1}(p)=(x, y)$, $x, y\in \h\cong \h^*$, contains a point with coordinates $x_1=y_2=0$ hence is uniquely determined by $p$. We can identify this orbit as well as $p$ with a point in the cone over the Grassmanian $Gr(2,n)$, and let $U=\langle x, y \rangle\subset \C^n$ be the two-dimensional plane representing this point. One can assume that $U$ does not contain eigenvectors of any $w \in W\setminus Z(W)$, hence $w(U)\ne U$ as required.

\end{proof}

\begin{remark}
    If $\dim \h = 2$ then $W \cap SO(2)$ acts trivially on $V \git \, \SL_2$ and so Lemma~\ref{lem:Wbaropenfree} is false in that case.  
\end{remark}

\begin{prop}\label{prop:Wbaropeninv}
    Assume $\dim \h > 1$ when $\Gamma = \T$ and $\dim \h > 2$ when $\Gamma = \SL_2$. The group $\overline{W}$ acts freely on an open subset of $V \git \, \Gamma$ with complement of codimension at least two. 
\end{prop}

\begin{proof}
For $\Gamma = \T$, $\pi^{-1}(0) = \h \times \{ 0 \} \cup \{ 0 \} \times \h^*$ has dimension $\dim \h$ and $\pi^{-1}(\overline{u})$ is a single closed free $\T$-orbit for any $\overline{u} \neq 0$, $\overline{u} \in V \git \, \Gamma$. It suffices to show that $\overline{W}$ acts freely on an open subset of $(V \git \, \Gamma) \setminus \{ 0 \}$ whose complement has codimension at least two (in either $(V \git \, \Gamma) \setminus \{ 0 \}$ or in $V \git \, \Gamma$). Similarly, for $\Gamma = \SL_2$, it is shown in the proof of Lemma~\ref{lem:stablerep}  that $\pi^{-1}(\overline{u})$ is a single closed free $\SL_2$-orbit for any $\overline{u} \neq 0$. Therefore,  it suffices to show that $\overline{W}$ acts freely on an open subset of $(V \git \, \Gamma) \setminus \{ 0 \}$ whose complement has codimension at least two. Let $U = V \setminus \pi^{-1}(0)$. 

For $u \in U$, let $\overline{u} = \pi(u)$. If $w \in W$ fixes $\overline{u}$ then $w(u) \in \Gamma \cdot u$ so there must exist $z \in \Gamma$ such that $zw(u) = u$. The element $z$ is unique because the stabilizer $\Gamma_u = \{ 1 \}$. Moreover, since $w$ has finite order, $z$ must also have finite order. In particular, it is a diagonalizable element of $\Gamma$. We must show that $\dim T_{\overline{u}} (V \git \, \Gamma) - \dim (T_{\overline{u}} (V \git \, \Gamma))^w$ is at least two. If this difference is zero then $w$ acts trivially on $V \git \, \Gamma$ since the variety is irreducible; in this case $w \in Z(W)$ by Lemma~\ref{lem:Wbaropenfree} and can be ignored. Thus, it suffices to show that the difference is always even. Equivalently, the number of non-trivial eigenvalues (counted with multiplicity) of $w$ acting on $T_{\overline{u}} (V \git \, \Gamma)$ is even.

If $\mf{k}$ is the Lie algebra of $\Gamma$ then 
\[
T_{\overline{u}} (V \git \, \Gamma) \cong (T_u V) / (\mf{k} \cdot u) = V / (\mf{k} \cdot u). 
\]
and the moment map for the Hamiltonian action of $\Gamma$ on $V$ is given by $\mu(u)(X) = \omega(X \cdot u, u)$, for $X \in \mf{k}$. Then $(d_u \mu)(v)(X) = 2 \omega(X \cdot u, v)$ which implies that $\Ker (d_u \mu) = \mf{k}^{\perp}$ (this is a general property of moment maps). Identifying $\mf{k}$ with the Lie algebra of $\Gamma \times W$, there is a $\langle wz \rangle$-equivariant short exact sequence 
\[
0 \to \mf{k}^{\perp} / \mf{k} \to V / (\mf{k} \cdot u)  = T_{\overline{u}} (V \git \, \Gamma) \stackrel{d_u \mu}{\longrightarrow} \mf{k}^* \to 0. 
\]
Now $\mf{k}^{\perp} / \mf{k}$ is a symplectic vector space on which $zw$ acts symplectically, so it is certainly the case that $\dim \mf{k}^{\perp} / \mf{k} - \dim (\mf{k}^{\perp} / \mf{k})^{wz}$ is even. The action of $zw$ on $\mf{k}^*$ equals the action of $z$. Since $z$ is semi-simple, we can choose a maximal torus $\T \subset \Gamma$ containing $z$. The fact that every root in the root space decomposition of $\mf{k}$ with respect to $\T$ appears with its negative implies that $\dim \mf{k}^* - \dim (\mf{k}^{*})^{z}$ is even too.

\end{proof}

For the remainder of this article, we assume:
\begin{center}
     $\dim \h > 1$ when $\Gamma = \T$ and $\dim \h > 2$ when $\Gamma = \SL_2$.
\end{center}

\subsection{Arbitrary parameters}

In this section we assume that the parameters $t,c$ are arbitrary scalars. Recall that $Z(W)$ is cyclic (of order $\ell$) since $\h$ is assumed to be irreducible. Let $\ezero, \ds, \elminusone \in \C Z(W)$ be the idempotent basis of this group algebra, where $\eo = \ezero$ is the symmetrizing idempotent. 

In the case of a real reflection group $W$ we have $\ell =1$ unless $W$ has an element acting as $- \text{Id}$ on $\mathfrak h$, in which case $\ell=2$. Then the idempotents are $\eo=\frac12(\text{Id} + (-\text{Id}))$ and $\eo_1 =\frac12(\text{Id} - (-\text{Id}))$.


By assumption (III), the conjugation action of $Z(W)$ on $H_{t,c}^{\Gamma}$ is trivial and hence $\C Z(W)$ is central in $H_{t,c}^{\Gamma}$.

The following is an immediate consequence of this fact. 

\begin{lem}\label{lem:aidecomp}
	$H_{t,c}^{\Gamma} = \bigoplus_{i = 0}^{\ell-1} \ei H_{t,c}^{\Gamma} \ei$ as an algebra and hence $H_{t,c}^{\Gamma} / \langle 1 - \ei \rangle \cong \ei H_{t,c}^{\Gamma} \ei$.
\end{lem} 

Recall \cite{BrownCohomology},  \cite[Chapter~2]{Karpilovsky} that a $2$-cocycle on $\overline{W}$ with values in the trivial $\overline{W}$-module $\Cs$ is a map $f \colon \overline{W} \times \overline{W} \to \Cs$ such that 
$$
f(g_2,g_3) f(g_1,g_2 g_3) = f(g_1 g_2, g_3) f(g_1,g_2)
$$
for all $g_1, g_2, g_3 \in \overline{W}$. If $\alpha\colon \overline{W}\to \Cs$ is a 1-cochain then
$d \alpha\colon \overline{W} \times \overline{W} \to \Cs$ given by   
\[
 (d\alpha)(g,h) := \alpha(h) \alpha(gh)^{-1} \alpha(g)
\]
for all $g, h \in \overline{W}$  is a 2-cochain. 
For any $2$-cocycle $f$ of $\overline{W}$, we can form a twisted group algebra $\C_f \overline{W}$ with multiplication $u \cdot w = f(u,w) uw$. If the $2$-cocycles $f,f'$ differ by a $2$-coboundary i.e. $f' = f  (d \alpha)$, then the twisted group algebras are isomorphic: $\C_{f'}  \overline{W} \iso \C_{f} \overline{W}$ is given by $u \mapsto \alpha(u) u$ for $u\in \overline{W}$ and extended linearly to $\C_f \overline{W}$. Thus, up to isomorphism, $\C_f \overline{W}$ only depends on $[f] \in H^2(\overline{W},\Cs)$ (see \cite[Chapter~2, Lemma~1.1]{Karpilovsky}).

Choose coset representatives $w_1, \ds, w_{\ell}$ of $Z(W)$ in $W$. 

\begin{lem}\label{lem:fcocycle}
	Define $f_i \colon \overline{W} \times \overline{W} \to \Cs$ by 
	$$
	w_a w_b \ei = f_i(\bar{w}_a , \bar{w}_b) w_c \ei
	$$
	if $w_a w_b = z_{ab} w_c$ in $W$ with $z_{ab} \in Z(W)$. Then $f_i$ is a $2$-cocycle and $(\C W) \ei \cong \C_{f_i} \overline{W}$. 
\end{lem}

\begin{proof}
	We will check directly that $f_i$ is a $2$-cocycle, but then explain conceptually why this is the case. If $\chi_i \colon Z(W) \to \Cs$ is the character defined by $\ei$ then $f_i(\bar{w}_a , \bar{w}_b) = \chi_i(z_{ab})$. Let $w_a w_b = z_{ab} w_x$ and $w_b w_c = z_{bc} w_y$. Note that 
	$$
	z_{ay} w_u := w_a w_y = w_a(z_{bc}^{-1} w_b w_c) = z_{ab} z_{bc}^{-1} w_x w_c. 
	$$
	Then $f_i(\bar{w}_b,\bar{w}_c) f_i(\bar{w}_a,\bar{w}_b \bar{w}_c) = \chi_i(z_{bc}) \chi_i(z_{ay})$ and 
	$$
	f_i(\bar{w}_a \bar{w}_b, \bar{w}_c) f_i(\bar{w}_a,\bar{w}_b) = \chi_i(z_{ay} z_{bc} z_{ab}^{-1}) \chi_i(z_{ab}) = \chi_i(z_{bc}) \chi_i (z_{ay}). 
	$$
 This implies that $f_i(\bar{w}_b,\bar{w}_c) f_i(\bar{w}_a,\bar{w}_b \bar{w}_c) = f_i(\bar{w}_a \bar{w}_b, \bar{w}_c) f_i(\bar{w}_a,\bar{w}_b)$. 
\end{proof}

As explained in \cite[Theorem~3.12]{BrownCohomology}, there is a canonical bijection between elements of $H^2(\overline{W},Z(W))$ and isomorphism classes of central extensions $1 \to Z(W) \to G \to \overline{W} \to 1$. In particular, $W$ itself corresponds to a class $[h] \in H^2(\overline{W},Z(W))$. The explicit construction given in \cite{BrownCohomology} shows that $f_i = \chi_i \circ h$. It follows directly that $f_i$ is a $2$-cocycle. 

Note that if we choose new coset representatives $w_a' = z_a w_a$, with corresponding $2$-cocycle $h'$, and define $\alpha \colon \overline{W} \to Z(W)$ by $\alpha(\bar{w}_a) = z_a$ then $h' = h (d \alpha)$. Therefore, $[h] = [h']$ in $H^2(\overline{W},Z(W))$. Hence, setting $\alpha_i = \chi_i\circ \alpha \colon \overline{W} \to \Cs$, we deduce that $\C_{f_i'} \overline{W} \cong \C_{f_i} \overline{W}$ is independent of the choice of coset representatives.

\begin{lem}\label{lem:greizero}
The ring $\ei H_{t,c}^{\Gamma} \ei$ contains $\C_{f_i} \overline{W}$ and $\gr (\ei H_{t,c}^{\Gamma} \ei) \cong P^{\Gamma} \rtimes \C_{f_i} \overline{W}$. 
\end{lem}

\begin{proof}
The first claim follows from Lemma~\ref{lem:fcocycle}. 

Note that $0 \to \mc{F}_{j-1} H_{t,c}^{\Gamma} \to \mc{F}_j H_{t,c}^{\Gamma} \to \gr_j H_{t,c}^{\Gamma} \to 0 $ is a sequence of $\C Z(W)$-modules. This implies that $\gr (\ei H_{t,c}^{\Gamma} \ei) \cong \ei (P^{\Gamma} \rtimes W) \ei$. Then, once again by Lemma~\ref{lem:fcocycle}, $\ei (P^{\Gamma} \rtimes W) \ei = P^{\Gamma} \rtimes \C_{f_i} \overline{W}$.  
\end{proof}

Recall that a (typically, noncommutative) ring $R$ is called {\it prime} if the equality $I J =0$ for two-sided ideals $I$ and $J$ implies that $I=0$ or $J=0$.

\begin{cor}\label{cor:primring}
	The ring $\ei H_{t,c}^{\Gamma} \ei$ is prime. 
\end{cor}

\begin{proof}
By Lemma~\ref{lem:greizero}, it suffices to prove that the algebra $P^{\Gamma} \rtimes \C_{f_i} \overline{W}$ is prime. But this is precisely the statement of \cite[Corollary~12.6]{Passman} since the group $\overline{W}$ acts by outer automorphisms on the commutative ring $P^{\Gamma}$. 
\end{proof}

We will usually focus on $\eo H^{\Gamma}_{t,c} \eo$ rather than $H^{\Gamma}_{t,c}$. The main reason for this is that Corollary~\ref{cor:primring} says that $\eo H^{\Gamma}_{t,c} \eo$ is prime. The ring $H^{\Gamma}_{t,c}$ is prime if and only if $H^{\Gamma}_{t,c} = \eo H^{\Gamma}_{t,c} \eo$, which is the case precisely when $Z(W) = 1$. 

We recall that a noetherian algebra $A$ is \textit{Auslander--Gorenstein} if it has both finite left and right injective dimensions and for each finitely generated (left or right) $A$-module $M$ the following holds: for every integer $i \ge 0$ and every submodule $N$ of $\Ext^i_{A}(M, A)$, we have $\Ext^j_{A}(N, A) = 0$ for all $j < i$. As noted in \cite{BrownChangtong}, by a theorem of Zaks, if a noetherian algebra has finite right and left injective dimensions, then these are equal.

The \textit{grade} of a finitely generated left $A$-module $M$ is defined  to be
\[
j_A(M) = \min \{ i \in \Z_{\ge 0} \, | \, \Ext^i_A(M,A) \neq 0 \}. 
\]
Assume that $A$ has finite Gelfand--Kirillov dimension. Then $A$ is said to be (GK) \textit{Cohen--Macaulay} if 
\[
j_A(M) + \GKdim (M) = \GKdim (A),
\]
for all finitely generated left $A$-modules $M$. 

Now let $A = P^{\Gamma} \rtimes W$. 
Consider $A$ as the left $P^{\Gamma}$-module. Then 
the space $\Hom_{P^{\Gamma}}(A,P^{\Gamma})$ is a left $A$-module via $(a \star \phi)(b) = \phi(ba)$ for $a,b \in A$ and $\phi \in \Hom_{P^{\Gamma}}(A,P^{\Gamma})$.

\begin{lem}\label{lem:HomAP0iso}
    There is an isomorphism of left $A$-modules $\sigma \colon {}_A A \iso \Hom_{P^{\Gamma}}({}_{P^{\Gamma}} A,P^{\Gamma})$ given by 
    \[
    \sigma(f \o u) (b) = b_{u^{-1}} u^{-1}(f), \quad \forall \ f \in P^{\Gamma}, u \in W,  
    \]    
    where
$b=\sum_{w \in W} b_w \o w \in A$, $b_w \in P^\Gamma$.
\end{lem}

\begin{proof}
    It is immediate that $\sigma(f \o u) \in \Hom_P({}_P A,P)$. We  check that it is a morphism of left $A$-modules. Let $f_i \in P^{\Gamma}, u_i \in W$. Then $(f_1 \o u_1) (f_2 \o u_2) = f_1 u_1(f_2) \o u_1 u_2$. So 
    \begin{align*}
        \sigma((f_1 \o u_1) (f_2 \o u_2)) (b) & = \sigma(f_1 u_1(f_2) \o u_1 u_2)(b) \\
        & = b_{u_2^{-1} u_1^{-1}} u_2^{-1} (u_1^{-1}(f_1 u_1(f_2))) \\
        & = b_{u_2^{-1} u_1^{-1}} u_2^{-1}(u_1^{-1}(f_1) f_2) \\
        & = b_{u_2^{-1} u_1^{-1}} (u_2^{-1}u_1^{-1}(f_1)) u_2^{-1}(f_2),
    \end{align*}
 and 
    \begin{align*}
        (f_1 \o u_1) \star \sigma( f_2 \o u_2) (b) & = \sigma(f_2 \o u_2)(b (f_1 \o u_1)) \\
        & =  \sigma(f_2 \o u_2)\left( \sum_{w \in W} b_w w(f_1) \o w u_1 \right) \\
        & = b_{u_2^{-1} u_1^{-1}} (u_2^{-1}u_1^{-1}(f_1)) u_2^{-1}(f_2).
    \end{align*} 
\end{proof}

\begin{lem}\label{lem:gradeAM}
Let $M$ be a finitely generated left $A$-module. Then $j_A(M) = j_{P^{\Gamma}}(M)$.
\end{lem}

\begin{proof}
    Since $A$ is free of finite rank over $P^{\Gamma}$, the Eckmann-Shapiro Lemma, together with Lemma~\ref{lem:HomAP0iso}, says that 
    \[
    \Ext^i_{P^{\Gamma}}(M,P^{\Gamma}) = \Ext^i_A(M,\Hom_{P^{\Gamma}}(A,P^{\Gamma})) = \Ext^i_A(M,A). 
    \]
    The lemma follows. 
\end{proof}

\begin{thm}\label{thm:AuslanderGorensteinH0}
	The algebras $H_{t,c}^{\Gamma}$ and $\ei H_{t,c}^{\Gamma} \ei$ are Auslander--Gorenstein and (GK)  Cohen--Macaulay. These algebras have Gelfand--Kirillov dimension $2 \dim \h - \dim \Gamma$. 
\end{thm}

\begin{proof}
	By Proposition~\ref{prop:invariantringGorenstein}, the ring $P^{\Gamma}$ is a commutative noetherian Gorenstein domain. Therefore it is Auslander--Gorenstein. It follows from \cite{YiGlasgow} that $\gr H_{t,c}^{\Gamma} = P^{\Gamma} \rtimes W$ is Auslander--Gorenstein. Hence $H_{t,c}^{\Gamma}$ is also Auslander--Gorenstein by \cite[Theorem 3.9]{BjorkAuslander}. This implies that each summand $\ei H_{t,c}^{\Gamma} \ei$ is also Auslander--Gorenstein.
	
Since the generic $\Gamma$-orbit in $\h \times \h^*$ is closed by Lemma~\ref{lem:stablerep}, the algebra 
\[
P^{\Gamma} = \C\left[ (\h \times \h^*) \git \, \Gamma \right] 
\]
has Krull (and hence Gelfand--Kirillov) dimension $2 \dim \h - \dim \Gamma$. By \cite[Proposition~8.1.14]{MR}, it follows that both $P^{\Gamma} \rtimes W$ and $H_{t,c}^{\Gamma}$ also have GK-dimension $2\dim \h - \dim \Gamma$. 

Finally, we check that $H_{t,c}^{\Gamma}$ and $\ei H_{t,c}^{\Gamma} \ei$ are Cohen--Macaulay. By \cite[Proposition 8.6.5]{MR}, it suffices to show that $A = P^{\Gamma} \rtimes W$ and $\ei A \ei$ are Cohen--Macaulay. Note that they have finite GK-dimension by the previous paragraph. The commutative ring $P^{\Gamma}$ is Cohen--Macaulay by Hochster's Theorem \cite[Theorem~6.4.2]{CohMac}. Thus, if $M$ is a finitely generated $A$-module, it is finitely generated over $P^{\Gamma}$ and applying Lemma~\ref{lem:gradeAM} gives
\[
\GKdim (A) = \GKdim (P^{\Gamma}) = \GKdim_{P^{\Gamma}} (M) + j_{P^{\Gamma}}(M) = \GKdim_{A}(M) + j_A(M).
\]
Thus, $A$ is Cohen--Macaulay. Each $\ei$ is central in $A$ and $\ei A \ei$ has the same GK-dimension as $A$. Therefore, if $M_i$ is a finitely generated $\ei A \ei$-module, $\Ext_A^r(M_i,A) = \Ext^r_{\ei A \ei}(M_i, \ei A \ei)$ and $\GKdim_{\ei A \ei} (M_i) = \GKdim_A(M_i)$. Thus, 
\[
j_{\ei A \ei}(M_i) + \GKdim_{\ei A \ei} (M_i) = j_A(M_i) + \GKdim_{A}(M_i) =  \GKdim (A) = \GKdim (\ei A \ei),
\]
showing that $\ei A \ei$ is Cohen--Macaulay too. 
\end{proof}

We note for later use:

\begin{cor}\label{cor:CMmodule}
    $\eo H^{\Gamma}_{c} e$ is a Cohen--Macaulay $eH^{\Gamma}_{c} e$-module, where $e \in \C W$ denotes the symmetrizing idempotent. 
\end{cor}

\begin{proof}
The proof is identical to the proof of \cite[Theorem~1.5(ii)]{EG}, using the fact that $P^{\Gamma}$ is a Cohen--Macaulay ring and a finitely generated module over the Gorenstein ring $P^{\Gamma \times W}$ (see Proposition~\ref{prop:invariantringGorenstein}). 
\end{proof}

\subsection{Double centralizer property}\label{doublecsection}

Let $A$ and $B$ be rings, and let $M= {_A}M_B$ be an $(A,B)$-bimodule. Then there are natural multiplication homomorphisms
$$
\varphi\colon B^\mr{op} \to \End_{A} M, \quad 
\psi\colon A \to \End_{B} M, 
$$
given by, respectively, right and left multiplication. 

\begin{lem}\label{cent1}
Let $e$ be an idempotent in $A$, and let $B= e A e$, $M=Ae$. Then $\varphi$ is an isomorphism.
\end{lem}
\begin{proof}
We have $\varphi(b)(e) = e b =b$, which implies that $\varphi$ has zero kernel. 

Let now $f \in \End_A M$. Let $b=f(e) \in M$. Hence $b = a e$ for some $a \in A$. Note that
$$
b = f(e) = f(e^2) = e f(e) = e b = e a e.
$$
Hence $b=e b e \in B$. 

Let us consider the image $f(m)$ for an arbitrary element $m\in M$. Let $m = x e$, where $x \in A$. We have
$$
f(m) = f(x e) = x f(e) = x b = \varphi(b) (m)
$$
since $x e b = x b$. It follows that $f =\varphi(b)$ and therefore $\varphi$ is surjective.
\end{proof}

By taking $A= \eo H^{\Gamma}_{t,c} \eo$ and $e$ the trivial  idempotent in $A$ we get by Lemma \ref{cent1} that
$$
\End_{\eo H^{\Gamma}_{t,c} \eo}(\eo H^{\Gamma}_{t,c} e) \cong (e H^{\Gamma}_{t,c}e)^{\rm op}. 
$$


The next theorem shows that under these settings $\psi$ is also an isomorphism so that the double centralizer property holds for the bimodule $\eo H^{\Gamma}_{t,c} e$. 

\begin{thm}\label{thm:doublecetralHGamma}
We have an isomorphism 
	$\End_{e H^{\Gamma}_{t,c} e}(\eo H^{\Gamma}_{t,c} e) \cong \eo H^{\Gamma}_{t,c} \eo$. 
\end{thm}

\begin{proof}
First, we consider the case $t=c=0$. Let $S$ be a normal domain with an action of $\overline{W}$ such that the corresponding action on $\Spec S$ is free on an open subset with complement of codimension at least 2. Then the left multiplication map $\psi\colon S \rtimes \overline{W} \to \End_{eSe}S$ is an isomorphism by \cite[Proposition~E.2.2(2)]{BonnafeRouquier}. 

Now let $S= \C[\h \times \h^*]^{\Gamma}$. The variety $(\h \times \h^*) \git \, \Gamma$ is normal by \cite[\S~3.3, Satz~1]{Kraft}. The action of $\overline{W}$ on $\Spec S$ is free on an open subset with complement of codimension at least $2$ by Proposition~\ref{prop:Wbaropeninv}. Therefore, the statement follows for zero parameters by the previous paragraph since $\eo H^{\Gamma}_{0,0} \eo = H^{\Gamma}_{0,0} \eo \cong S \rtimes \overline{W}$.



Then, using filtration arguments analogously to the proof of \cite[Theorem~1.5(iv)]{EG}, 
the statement follows for all parameters $(t,c)$. 
\end{proof}

We note that the Satake isomorphism holds in this situation. The proof is identical to that of the usual Satake isomorphism for symplectic reflection algebras; see \cite[Theorem~1.7.3]{BellamySRAlecturenotes}. 

\begin{cor}\label{cor:Satakeinv}
The map $z \mapsto e z$ is an isomorphism $Z(\eo H^{\Gamma}_{t,c} \eo) \iso Z(e H^{\Gamma}_{t,c} e)$.  
\end{cor}

Since $e H^{\Gamma}_{0,c} e$ is commutative, Corollary~\ref{cor:Satakeinv} implies the following. 

\begin{cor}\label{cor:Satakeinvteq0}
  When $t = 0$, there is isomorphism
    \[
  Z(\eo H^{\Gamma}_{0,c} \eo) \iso e H^{\Gamma}_{0,c} e, \quad z \mapsto e z.  
  \]
\end{cor}

\subsection{The centre of $H_{t,c}^{\mf{gl}(n)}$ at $t \neq 0$}

We note that it is not true that the centre of $H_{t,c}^{\mf{gl}(n)}$ is trivial when $t \neq 0$. Indeed, it was shown in \cite[Section~7]{Feiginangular} that $Z(H_{1,c}^{\mf{gl}(n)}(\s_n)) = \C[\eu_{c}]$. We show that a similar result holds for any complex reflection group. 

By Proposition~\ref{prop:centreskewgroup1} below, the centre of $P^{\T} \rtimes W$ equals $P^{\T \times W} \o \C Z(W)$. It is a Poisson algebra because $H_{t,c}^{\T}$ is a quantization of $P^{\T} \rtimes W$ for any $(t,c)$ with $t \neq 0$. The Poisson structure on $P^{\T \times W} \o \C Z(W)$ is independent of the choice of $(t,c)$. 

If $B$ is a Poisson algebra then $\mr{Cas}(B) \subset B$ denotes the Poisson centre, consisting of all elements $b$ such that $\{ b , - \} = 0$ ($\mr{Cas}$ stands for "Casimirs'').

\begin{lem}\label{lem:Casimirfiniteextension}
    Let $A,B$ be $\C$-algebras. Assume $B$ is a Poisson domain and $A \subset B$ a Poisson subalgebra such that $B$ is a finite $A$-module. Then $\mr{Cas}(A) = \mr{Cas}(B) \cap A$. 
\end{lem}

\begin{proof}
This statement is implicit in the final paragraph of the proof of \cite[Proposition~7.2(2)]{PoissonOrders}, but we include the proof for the reader's convenience. It is clear that $\mr{Cas}(B) \cap A \subset \mr{Cas}(A)$. Let $p \in \mr{Cas}(A)$ and $u \in B$. Since $B$ is a finite $A$-module, $u$ is integral over $A$. Therefore, there exists a polynomial of minimal degree $\sum_{i = 0}^n a_i X^i$ with $a_n = 1$ and $a_i \in A$ such that $\sum a_i u^i = 0$. Then 
\[
0 = \left\{ p, \sum_{i = 0}^n a_i u^i \right\} = \sum_{i = 0}^n (\{ p, a_i \} u^i + a_i \{ p, u^i \}) = \left( \sum_{i = 0}^n i a_i u^{i-1} \right) \{ p, u \}. 
\]
By minimality of the polynomial and the fact that we are working over a field of characteristic zero, $\sum i a_i u^{i-1} \neq 0$ and hence $\{ p, u \} = 0$ because $B$ is a domain. Therefore, $p \in \mr{Cas}(B)$. 
\end{proof}

\begin{lem}\label{lem:PoissoncentreP0}
    The Poisson centre of $P^{\T \times W} \otimes \C Z(W)$ equals $\C[\mr{eu}_0] \o \C Z(W)$. 
\end{lem}

\begin{proof}
The idempotents $\ei \in \C Z(W)$ are central in $H_{1,c}^{\T}(W)$. This means that they are Poisson central in $P^{T \times W} \otimes \C Z(W)$. Hence the map $(f_i) \mapsto \sum_i f_i \o \ei$ is an isomorphism of Poisson algebras $(P^{\T \times W})^{\oplus \ell} \cong P^{\T \times W} \o \C Z(W)$. Therefore, it suffices to show that $\mr{Cas}(P^{\T \times W}) = \C[\mr{eu}_0]$. 

The Poisson structure on $P^{\T \times W}$ comes by restriction from the Possion structure on $P^{\T}$. In other words, $P^{\T \times W} \subset P^{\T}$ is a (finite) extension of Poisson algebras. Both these algebras are noetherian domains since they are subrings of invariants in $P$. Therefore, Lemma~\ref{lem:Casimirfiniteextension} says that $\mr{Cas}(P^{\T \times W})$ is the intersection of $\mr{Cas}(P^{\T})$ with $P^{\T \times W}$. Hence, it suffices to argue that $\mr{Cas}(P^{\T})= \C[\mr{eu}_0]$. This is precisely the content of the proof of \cite[Lemma~1]{Feiginangular}. 
\end{proof}

\begin{prop}\label{centrefortnot}
For $t \neq 0$, the centre of $H_{t,c}^{\mf{gl}(n)}$ equals $\C[\mr{eu}_{c}] \o \C Z(W)$.
\end{prop}

\begin{proof}
    It is clear that $\C[\mr{eu}_{c}] \o \C Z(W)$ belongs to the centre of $H_{t,c}^{\mf{gl}(n)}$. Now assume $f \in Z(H_{t,c}^{\mf{gl}(n)})$. Then the symbol $\sigma(f)$ of $f$ belongs to the Poisson centre of 
    \[
    Z(\gr H_{t,c}^{\mf{gl}(n)}) = Z(P^{\T} \rtimes W) = P^{\T \times W} \otimes \C Z(W).
    \]
    The Poisson centre of $P^{\T \times W} \otimes \C Z(W)$ equals $\C[\mr{eu}_0] \o \C Z(W)$ by Lemma~\ref{lem:PoissoncentreP0}.  We argue by induction on the degree of $\sigma(f)$ that this means $f \in  \C[\mr{eu}_{c}] \o \C Z(W)$. If $\deg \sigma(f) = 0$ then $f \in \mc{F}_0 = \C Z(W)$ since $\mc{F}_{-1} = 0$. Therefore, we may assume $\deg \sigma(f) = k > 0$, where $k$ is even, and $g \in \C[\mr{eu}_c] \o \C Z(W)$ for all $g \in Z(H_{t,c}^{\mf{gl}(n)})$ of degree less than $k$. Since $\sigma(f)$ is homogeneous, $\sigma(f) = \sum_i \mr{eu}_0^{k/2} \o d_i \ei$ for some $d_i \in \C$. Then $f - \sum_i \mr{eu}_{c}^{k/2} \o d_i \ei \in \mc{F}_{k-1}$ is still central. By induction, $f - \sum_i \mr{eu}_{c}^{k/2} \o d_i \ei \in \C[\mr{eu}_{c}] \o \C Z(W)$ and hence so too is $f$.
\end{proof}

Proposition~\ref{centrefortnot}, together with Lemma~\ref{lem:aidecomp},  implies that the centre of $\eo H_{1,c}^{\mf{gl}(n)} \eo$ equals $\C[\eo \eu_{c}]$.

\section{Invariants at $t = 0$}\label{sec:invariantt=0}

In this section we consider the properties of the invariant rings at $t = 0$. For brevity, we write $H_{c} = H_{c}(W)$ for the rational Cherednik algebra at $t = 0$ and some fixed $c$.

\subsection{The centre at $t = 0$}

 We begin by considering the case $c = 0$. Let $U = (\h \o \h^*)^{\Gamma}$ and note that $W$ acts on $U$ by (II).  

\begin{prop}\label{prop:centreskewgroup1}
    	If $\Ker(W \to \GL(U)) = Z(W)$ then $Z(P^{\Gamma} \rtimes W) = P^{\Gamma \times W} \o \C Z(W)$.
\end{prop}

\begin{proof}
If we identify degree $2$ part of the $\N$-graded algebra $P$ with $\h \otimes \h^*$ then $U \subset P^{\Gamma}$. Let $a = \sum_w f_w \o w$ be in the centre of $P^{\Gamma} \rtimes W$. Then, for $u \in U$, 
\[
[a,u] = \sum_w f_w (w^{-1}(u) - u) \o w.
\]
For this to be zero, we must have $w^{-1}(u) = u$ for all $w$ with $f_w \neq 0$ because $P^{\Gamma}$ is a domain. Since $\Ker(W \to \GL(U)) = Z(W)$, it follows that $a = \sum_{w \in Z(W)} f_w \o w$. Moreover, $[a,u] = 0$ for all $u \in W$ implies that $f_w \in P^{\Gamma \times W}$. 
	
	On the other hand, it is clear that $a = \sum_{w \in Z(W)} f_w \o w$ with $f_w \in P^{\Gamma \times W}$ belongs to the centre. 
	
\end{proof}

\begin{remark}
In the case $\Gamma = \T$, the space $U$ equals $\h \otimes \h^* = \mf{gl}(n)$ and $\Ker (W \to \GL(\mf{gl}(n)))$ equals $Z(W)$. This follows from the fact that $W \cap \Cs \Id_{\h} = Z(W)$ in $\GL(\h)$. In the case $\Gamma = \SL_2$, we identify $U = \mf{so}(n) \subset \mf{gl}(n)$ and Corollary~\ref{dim2} says that $\Ker(W \to \GL(\mf{so}(n))) = Z(W)$ when $\dim \h > 2$. 
\end{remark}

Proposition~\ref{prop:centreskewgroup1} implies that $Z(P^{\Gamma} \rtimes W) = P^{\Gamma \times W} \o \C Z(W)$ has basis $\eo = \ezero, \ds, \elminusone$ as a module over the algebra $P^{\Gamma \times W}$. 

\begin{lem}\label{lem:centreinclude}
    The ring $Z(H_{c}^{\Gamma})$ is a $Z(H_{c})^{\Gamma}$-algebra and $Z(W) \subset Z(H_{c}^{\Gamma})$.
\end{lem}

\begin{proof}
For the first claim, it suffices to note that the inclusion $Z(H_{c}) \subset H_{c}$ restricts to an inclusion of algebras $Z(H_{c})^{\Gamma} \subset Z(H_{c}^{\Gamma})$. The inclusion $Z(W) \subset Z(H_{c}^{\Gamma})$ follows from the fact that the image of the conjugation action of $Z(W)$ on $H_{c}$ is assumed by (III) to be in $\Gamma$. This means that for $z \in Z(W)$ and $h \in H^{\Gamma}_{c}$, $z hz^{-1} = h$.  
\end{proof}

\begin{lem}\label{lem:grcentrezero}
    $\gr Z(H_{c})^{\Gamma} = P^{\Gamma \times W}$.
\end{lem}

\begin{proof}
Since $\Gamma$ is reductive Lemma~\ref{lem:gradedin}(iii) applies and so $\gr Z(H_{c})^{\Gamma}=(\gr Z(H_{c}))^{\Gamma}$.
Then \cite[Theorem 3.3]{EG} states that $\gr Z(H_{c})=Z(H_0)$, completing the proof. 
\end{proof}

We define $Y_c = Y_{c}(W, \Gamma) := \Spec \, Z(H_{c})^{\Gamma}$. 

\begin{thm}\label{ratsing}
The affine scheme $Y_{c}$ is a normal irreducible variety. It is Gorenstein with rational singularities.
\end{thm}

\begin{proof}
    The ring $Z(H_{c})$ is an integrally closed domain since its associated graded is $\C[\h \times \h^*]^W$, which is integrally closed. It is well-known that taking invariants with respect to a reductive group preserves these properties e.g. \cite[\S~3.3, Satz~1]{Kraft}. 

	Since $\Spec Z(H_{c})$ has rational singularities (see the proof of \cite[Proposition~4.5]{IainSurvey}) and $Y_{c}$ is the categorical quotient of $\Spec Z(H_{c})$ by a reductive group, \cite{Boutot} shows that $Y_{c}$ has rational singularities. 
	
	The algebra $Z(H_{c})^{\Gamma}$ is a filtered algebra with associated graded equal to $\C[\h \times \h^*]^{\Gamma \times W}$ by Lemma \ref{lem:grcentrezero}. Therefore it follows from Proposition~\ref{prop:invariantringGorenstein} and \cite[Theorem~3.9]{BjorkAuslander} that $Y_{c}$ is Gorenstein. 
\end{proof}


The following result was the main motivation for this work. 

\begin{thm}\label{assgrade}
    For any $c$, $\gr Z(H_{c}^{\Gamma}) = Z(P^{\Gamma} \rtimes W)$. 
\end{thm}

\begin{proof}
First, Corollary~\ref{cor:grglalg} and Proposition~\ref{prop:centreskewgroup1} imply that 
\[
Z(\gr H_{c}^{\Gamma}) = Z(P^{\Gamma} \rtimes W) = P^{\Gamma \times W} \o \C Z(W).
\]
Next, Lemma~\ref{lem:gradedin}(ii) says that $\mathrm{gr}\, Z(H_{c}^{\Gamma})$ is contained in $Z(\gr H_{c}^{\Gamma})$. Therefore, the claim follows if we can show that $Z(P^{\Gamma} \rtimes W)$ is contained in $\mathrm{gr}\, Z(H_{c}^{\Gamma})$; we need to show that $P^{\Gamma \times W}$ and $Z(W)$ are contained in $\mathrm{gr}\, Z(H_{c}^{\Gamma})$. 

Lemma~\ref{lem:centreinclude} implies that $Z(W) \subset \mc{F}_0 Z(H_{c}^{\Gamma})$ and hence $Z(W)$ is contained in $\mathrm{gr}\, Z(H_{c}^{\Gamma})$ because $\mc{F}_{-1} Z(H_{c}^{\Gamma}) = 0$ (by definition, $\mc{F}_{-1} H_{c} = 0$). Similarly,  Lemma~\ref{lem:centreinclude} says that $Z(H_{c})^{\Gamma}$ is contained in $Z(H_{c}^{\Gamma})$. Hence, Lemma~\ref{lem:gradedin}(i) implies that $\gr Z(H_{c})^{\Gamma}$ is contained in $\gr Z(H_{c}^{\Gamma})$. By Lemma~\ref{lem:grcentrezero}, this means that $P^{\Gamma \times W}$ also belongs to $\gr Z(H_{c}^{\Gamma})$, as required.
\end{proof}

\begin{cor}\label{centregln}
There is equality of commutative rings
\[
Z(H_{c}^{\Gamma}) = Z(H_{c})^{\Gamma} \o \C Z(W).
\]
In particular, as a $Z(H_{c})^{\Gamma}$-module, $Z(H_{c}^{\Gamma}) = \bigoplus_{i = 0}^{\ell-1} Z(H_{c})^{\Gamma} \ei$.
\end{cor}

\begin{proof}
Since the $\ei$ are a basis of $\C Z(W)$, it suffices to show that $Z(H_{c}^{\Gamma})$ equals the direct sum $\bigoplus_{i = 0}^{\ell-1} Z(H_{c})^{\Gamma} \ei$, with each $Z(H_{c})^{\Gamma} \ei$ free of rank one over $Z(H_{c})^{\Gamma}$. 

At $c = 0$, the space $Z(H_{0})^{\Gamma} \ei = P^{\Gamma \times W} \ei$ is clearly free of rank one over $P^{\Gamma \times W}$. Since 
\[
\gr (Z(H_{c})^{\Gamma} \ei) = P^{\Gamma \times W} \ei
\]
(as in the proof of Lemma~\ref{lem:greizero}) it follows that $Z(H_{c})^{\Gamma} \ei$ is free of rank one over $Z(H_{c})^{\Gamma}$. The inclusion $\bigoplus_{i = 0}^{\ell-1} Z(H_{c})^{\Gamma} \ei \subset Z(H_{c}^{\Gamma})$ gives rise to an inclusion 
\[
P^{\Gamma \times W} \otimes_{\C} \C Z(W) = \bigoplus_{i = 0}^{\ell-1} P^{\Gamma \times W} \ei = \gr\left( \bigoplus_{i = 0}^{\ell-1} Z(H_{c})^{\Gamma} \ei\right) \subset \gr Z(H_{c}^{\Gamma}) = P^{\Gamma \times W} \otimes_{\C} \C Z(W).
\]
We deduce from Lemma~\ref{lem:gradedin}(i) that $\bigoplus_{i = 0}^{\ell-1} Z(H_{c})^{\Gamma} \ei = Z(H_{c}^{\Gamma})$. 
\end{proof}

Corollary~\ref{centregln} implies that $Z(H_{c}^{\Gamma})$ is a domain if and only if $Z(W) = 1$. 

\begin{lem}\label{lem:Zgammaei}
    For each $i = 0, \ds, \ell-1$, there is an isomorphism $Z(H_{c})^{\Gamma} \iso Z(\ei H^{\Gamma}_{c} \ei)$, given by $z \mapsto \ei z$. Hence each irreducible component of $\Spec Z(H_{c}^{\Gamma})$ is isomorphic to $\Spec Z(H_{c})^{\Gamma}$.  
\end{lem}

Thus, $\Spec Z(H_{c}^{\Gamma})$ is the disjoint union of $\ell$ copies of $Y_{c}$. 

\begin{proof}
We note that $Z(\ei H^{\Gamma}_{c} \ei) = \ei Z( H^{\Gamma}_{c}) \ei$ because $\ei \in H^{\Gamma}_{c}$ is a central idempotent. Therefore, Corollary~\ref{centregln} implies that the morphism is surjective. On the other hand, we know that the map $Z(H_{c})^{\Gamma} \to e H_{c}^{\Gamma} e$ given by $z \mapsto e z$ is an isomorphism because it is the restriction of the usual Satake isomorphism to $Z(H_{c})^{\Gamma}$. In particular, it is injective. Now, if $z \in Z(H_{c})^{\Gamma}$ with $\ei z = 0$ then $0 = e(\ei z) = e z$ and hence $z = 0$.    
\end{proof}

In particular, $z \mapsto \eo z$ is an isomorphism $Z(H_{c})^{\Gamma} \iso Z(\eo H_{c}^{\Gamma} \eo)$.

\subsection{Simple modules}

In this section we consider the structure of a generic simple $\eo H_{c}^{\Gamma}\eo$-module. Let $k$ be an algebraically closed field and recall that a $k$-algebra is \textit{affine} if it is finitely generated as an algebra over $k$. Let $A$ be a prime affine $k$-algebra that is a finite module over its centre $Z(A)$. We use here the theory of polynomial identity algebras. Instead of recalling the relevant definitions, we refer the reader to \cite[Chapter~13]{MR} for a concise summary. Since $A$ is a finite module over its centre $Z(A)$, it is a PI-algebra \cite[Corollary~13.1.3]{MR}. We write $\mr{P.I.deg}(A)$ for the PI-degree of the prime PI-algebra $A$, which is the smallest $n \in \mathbb N$ such that there is an embedding of $A$ into the algebra of $n$ by $n$ matrices over a commutative ring $B$. It is known that the degree of a minimal degree monic polynomial which vanishes identically on elements of $A$ is $2n$ \cite{MR}.

\begin{defn}\label{defn:azumayalocus}
Let $M \subset A$ be a maximal (two-sided) ideal and $\mf{m} = M \cap Z(A)$. We recall that $\mf{m}$ belongs to \textit{the Azumaya locus} of $\Spec Z(A)$ if one of the following equivalent conditions hold 
\begin{enumerate}
    \item[(a)]$M$ is the unique maximal ideal of $A$ such that $M \cap Z(A) = \mf{m}$;  
    \item[(b)] there is a unique simple $A$-module $L$ whose annihilator in $Z(A)$ equals $\mf{m}$ (and thus, $\mr{ann}_A L = M$);
    \item[(c)] $A_{\mf{m}}$ is Azumaya over $Z(A)_{\mf{m}}$;
    \item[(d)] $A/M\cong {\rm Mat}_n(k)$ for some $n\in \N$.
\end{enumerate}
We say that $M$ is a \textit{regular maximal ideal} of $A$ if the equivalent conditions above hold. 
\end{defn}
For the proof and other equivalent characterizations of regular maximal ideal, see \cite[III, Theorem~1.6]{BrownGoodearlbook}. 
Note that $n$ from the property (d) equals the PI-degree of the algebra $A$ by the cited theorem.

The following lemma says that given a flat family of PI-algebras, where the centre of the algebras also varies flatly, the PI-degree of the algebras is constant. Hence, it needs only be computed for one fibre.  

\begin{lem}
    Let $B \subset Z(A)$ be a subalgebra and $\mf{p} \lhd B$ a (prime) ideal such that $A / \mf{p} A$ is prime. Assume that the map $Z(A) \to Z(A / \mf{p} A)$ is surjective. Then 
    \[
    \mr{P.I. deg} (A) = \mr{P.I. deg} (A / \mf{p} A).
    \]
\end{lem}

\begin{proof}
    Let $\overline{M}$ be a regular maximal ideal of $A / \mf{p} A$ and $M$ its preimage in $A$. It suffices to show that $M$ is a regular (maximal) ideal in $A$. Let $\overline{\mf{m}} = Z(A / \mf{p} A) \cap \overline{M}$ and write $\mf{m}$ for the preimage of $\overline{\mf{m}}$ in $Z(A)$. Then $\overline{M} = \overline{\mf{m}} (A / \mf{p} A)$ by \cite[Theorem III.1.6]{BrownGoodearlbook}. We claim that $\mf{m} = M \cap Z(A)$ and $M = \mf{m} A$. If this is indeed the case then $M$ is regular in $A$, again by \cite[Theorem III.1.6]{BrownGoodearlbook}.

    Let $z \in Z(A)$ such that $\overline{z} \in \overline{\mf{m}}$. Then $\overline{z} \in \overline{M}$ and $\overline{z} \in Z(A / \mf{p} A)$. Hence $z \in M$ and $M \cap Z(A) \subset \mf{m}$. But both are maximal ideals of $Z(A)$, thus $M \cap Z(A) = \mf{m}$. 

    Finally, we argue that $M = \mf{m} A$. Clearly, $\mf{m} A \subset M$. If $x \in M$ then there exist $\overline{z}_i \in \overline{\mf{m}}$ and $\overline{x}_i \in A/ \mf{p}A$ such that $\overline{x} = \sum_i \overline{z}_i \overline{x}_i$ in $A / \mf{p} A$. We choose lifts $z_i \in \mf{m}$ and $x_i \in A$. Then $x - \sum_i z_i x_i \in \mf{p} A$. But $\mf{p} A \subset \mf{m} A$ since $\mf{p} \subset \mf{m}$. Thus, $x \in \mf{m} A$. 
\end{proof}

\begin{cor}\label{cor:PIflatfamily}
    Let $A$ be a prime affine $k$-algebra, finite over $Z(A)$, and $B \subset Z(A)$ a subalgebra such that for all $\mf{m} \lhd B$ maximal the map $Z(A) \to Z(A(\mf{m}))$ is surjective, where $A(\mf{m}) := A / \mf{m} A$. Then 
    \[
    \mr{P.I. deg} (A) = \mr{P.I. deg} (A / \mf{m} A), \quad \forall \, \mf{m}.
    \]
\end{cor}

For a prime $k$-algebra $A$ finite over its centre $Z(A)$ let $L$ be a simple module. Then the annihilator $\mf{m}=\mr{ann}_{Z(A)} L$ is a maximal ideal in $Z(A)$, known as the support of $L$ (see \cite{BrownGoodearl} and Lemma \ref{lemma-on-support} below).

\begin{thm}\label{thm:PIsimplemodulesA}
Let $A$ be a prime affine $k$-algebra, finite over $Z(A)$. Let $e \in A$ be an idempotent such that 
\begin{enumerate}
    \item[(a)] $A \to \End_{eAe}(Ae)$ is an isomorphism,
    \item[(b)] $eAe$ is commutative; and
    \item[(c)] $Ae$ is a Cohen-Macaulay $eAe$-module. 
\end{enumerate}
If $Y = \Spec Z(A)$ then 
\begin{enumerate}
\item[(i)] The regular locus of $Y$ is contained in the Azumaya locus. 
    \item[(ii)] Let $L$ be a simple $A$-module whose support $\{ \mf{m} \}$ is a maximal ideal contained in the regular locus of $Y$. Then $L \cong Ae / \mf{m} A e$.  
    \end{enumerate} 
\end{thm}

\begin{proof}
Our proof follows the outline of the proof of \cite[Theorem~1.7]{EG}. As in \textit{loc. cit.}, we work geometrically. We note first that (a) and (b) imply that the morphism $Z(A) \to e A e$ given by $z \mapsto ez$ is an isomorphism; see the proof of \cite[Theorem~1.7.3]{BellamySRAlecturenotes}. We identify the two algebras via this isomorphism. The finitely generated $Z(A)$-module $Ae$ corresponds to a coherent sheaf $\mc{M}$ on $Y$, with $\Gamma(Y,\mc{M}) = Ae$. Then (c) says that $\mc{M}$ is a Cohen-Macaulay sheaf. 

Recall that the support of $\mc{M}$ is $\Supp \, \mc{M} = \{\mf{p} \in \Spec Z(A) \, | \, \mc{M}_\mf{p}\ne 0\}$. Then the dimension of $\mc{M}$ is
$\dim \Supp \, \mc{M}$, which equals $\dim Y$ since $eAe \cong Z(A)$ is a direct summand of $Ae$. Therefore, if $\mf{m} \lhd Z(A)$ is a maximal ideal contained in the regular locus of $Y$ then, as noted in the proof of \cite[Theorem~1.7]{EG}, $\mc{M}_{\mf{m}} = (A e)_{\mf{m}}$ is a free $\mc{O}_{Y,\mf{m}} = (eAe)_{_{\mf{m}}}$-module. In other words, $\mc{M}$ is a locally free coherent sheaf when restricted to the regular locus of $Y$. 

Next, (a) says that 
\[
A \to \End_{eAe}(Ae) \cong \End_{\mc{O}_Y}(\mc{M}) = \Gamma(Y,\mc{E}nd_{\mc{O}_Y}(\mc{M}))
\]
is an isomorphism. Thus, $A_{\mf{m}} \iso \Gamma(Y,\mc{E}nd_{\mc{O}_Y}(\mc{M}))_{\mf{m}}$. Since $\mc{E}nd_{\mc{O}_Y}(\mc{M})$ is a (quasi-)coherent $\mc{O}_Y$-module by \cite[Lemma~26.7.2(3)]{stacks} and $Y$ is an affine scheme, we have  $\Gamma(Y,\mc{E}nd_{\mc{O}_Y}(\mc{M}))_{\mf{m}} =\mc{E}nd_{\mc{O}_Y}(\mc{M})_{\mf{m}}$. The fact that $\mc{M}$ is coherent implies that 
\begin{equation}
    \label{localisation-matrix}
A_\mf{m}\cong
\mc{E}nd_{\mc{O}_Y}(\mc{M})_{\mf{m}} \cong \End_{\mc{O}_{Y,\mf{m}}}(\mc{M}_{\mf{m}}) = \End_{(eA e)_{_{\mf{m}}}}((A e)_{\mf{m}}) \cong \mr{Mat}_{d}((eA e)_{{\mf{m}}}),
\end{equation}
where the second isomorphism is \cite[Lemma~17.22.4]{stacks} and we have used that $(Ae)_{\mf{m}}$ is a free $(eAe)_{{\mf{m}}}$-module in the final isomorphism: 
\begin{equation}
    \label{dirsum}
(Ae)_{\mf{m}}= (eAe)_{{\mf{m}}}^{\oplus d}    
\end{equation}
 for some $d\in \N$.
It follows from \eqref{localisation-matrix} that we have ring isomorphisms 
\begin{equation}
    \label{matiso}
    A / \mf{m} A \cong A_{\mf{m}} / \mf{m} A_{\mf{m}} \cong \mr{Mat}_{d}(k).
\end{equation}

Note that $\mf{m}A$ is a maximal ideal in $A$ since the quotient $A/\mf{m}A$ is a matrix ring, which is simple. Note also that $\mf{m}A\cap Z(A) =\mf{m}$. Indeed, $\mf{m}A\cap Z(A) \supseteq\mf{m}$, and if $\mf{m}A\cap Z(A) = Z(A)$ then we would have $1\in \mf{m}A$ and hence $\mf{m}A=A$, which is not the case.
It follows by part (d) of Definition \ref{defn:azumayalocus} that $\mf{m}$ is in the Azumaya locus of $Y$ and that $d$ equals the PI-degree of $A$. This proves part (i).

In order to prove part (ii) we use uniqueness in part (b) of an 
equivalent characterization of Azumaya locus given in Definition~\ref{defn:azumayalocus}.
Thus, it is sufficient to show that the $A$-module $Ae/\mf{m}Ae$ is simple since its annihilator contains $\mf{m}$. 

It follows from \eqref{dirsum} that
the fiber of $\mc{M}$ at $\mf{m}$ is 
\[
(Ae)_\mf{m}/ \mf{m} (Ae)_\mf{m} \cong 
 k^d
\]
as a $(e A e)_\mf{m}/\mf{m}(eAe)_\mf{m}\cong k$-module.
We also have isomorphisms of $eAe$-modules 
\[
(Ae)_\mf{m}/ \mf{m} (Ae)_\mf{m}  \cong (Ae / \mf{m} A e)_{\mf{m}} \cong A e / 
 \mf{m} A e.  
\]
The action of $A$ on $Ae/\mf{m}Ae\cong k^d$ factors through $A/\mf{m}A$, which is isomorphic to $\mr{Mat}_{d}(k)$ by \eqref{matiso}. The statement follows since $k^d$ is necessarily a simple module for $\mr{Mat}_{d}(k)$. 
\end{proof}

We note the following lemma, which is a well-known consequence of Hilbert's finiteness theorem, but for which we don't know an explicit reference. 

\begin{lem}\label{lem:fginvariants}
Let $G$ be a complex reductive algebraic group, $Z$ an affine $\C$-algebra on which $G$ acts rationally and $M$ a finitely generated $G$-equivariant $Z$-module. Assume that the action of $G$ on $M$ is locally finite. Then $M^G$ is a finitely generated $Z^G$-module. 
\end{lem}

\begin{proof}
Since $G$ acts rationally on $Z$, this action is locally finite by e.g. \cite[II,~Lemma~2.4]{Kraft}. In particular, every finite-dimensional subspace of $Z$ is contained in some finite-dimensional $G$-submodule. Therefore, we may choose a finite-dimensional $G$-submodule $V$ of $Z$ such that $Z$ is generated as an algebra by $V$ i.e. there is a $G$-equivariant algebra surjection $S := \Sym \, V \to Z$. The algebra $Z$ is a semi-simple $G$-module since $G$ is assumed to be reductive; see Section~\ref{sec:setupinvariant}. This implies that taking invariants is exact and hence the map $S^G \to Z^G$ is also surjective. The action of $Z$ on $M$ lifts to $S$ and hence it suffices to show that $M^G$ is a finitely generated $S^G$-module. 

Since $M$ is also a sum of finite-dimensional $G$-modules, we can choose a finite-dimensional $G$-submodule $M_0$ of $M$ such that $M$ is generated by $M_0$ over $S$. This means that the multiplication map $S \o_{\C} M_0 \to M$ is surjective. It is also $G$-equivariant if $G$ acts diagonally on the tensor product. Again, since $G$ is reductive, this induces a surjection $(S \o_{\C} M_0)^G \to M^G$ and it suffices to prove that $(S \o_{\C} M_0)^G$ is finitely generated over $S^G$. Decomposing $M_0 = \bigoplus_{i} V_i \o_{\C} U_i$ into a direct sum of irreducible $G$-modules $V_i$, with multiplicity spaces $U_i$, we have $(S \o_{\C} M_0)^G = \bigoplus_{i} (S \o_{\C} V_i)^G \o_{\C} U_i$ and hence it suffices to show that each $(S \o_{\C} V_i)^G$ is a finitely generated $S^G$-module. But this is precisely the statement of \cite[II,~Zusatz~3.2]{Kraft}. 
\end{proof}

\begin{lem}\label{lem:PIdegreeskewgroup}
Let $Z$ be an affine, commutative, $\C$-algebra and a domain. Let $G$ be a finite group of automorphisms of $Z$ such that the action on $\Spec Z$ is generically free. Then the PI-degree of the prime PI-ring $Z \rtimes G$ equals $|G|$. 
\end{lem}

\begin{proof}
The ring $Z \rtimes G$ is PI since it is a finite module over its centre $Z^G$. It is prime because $Z$ is assumed to be a domain. Since $Z$ is affine, so too is $Z \rtimes G$. Therefore, as explained in the proof of \cite[III,~Theorem~1.6]{BrownGoodearlbook}, the PI-degree of $Z \rtimes G$ equals the supremum of the dimensions of simple $Z \rtimes G$-modules. Using \cite[Lemma~3.5]{BellPaegelow}, the simple modules can be described as follows. Choose $\mf{m} \lhd Z$ a maximal ideal and let $G_{\mf{m}}$ denote the stabilizer of the corresponding point of $\Spec Z$; $G_{\mf{m}}$ is the group of elements $g \in G$ with $g(\mf{m}) \subset \mf{m}$. Let $\lambda$ be an irreducible $G_{\mf{m}}$-module and set 
\[
V_{\mf{m}}(\lambda) := \Ind_{Z \rtimes G_{\mf{m}}}^{Z \rtimes G} \lambda,
\]
where $Z$ acts on $\lambda$ by $z \cdot v = \chi(z) v$ with $\chi \colon Z \to \C$ the homomorphism with kernel $\mf{m}$. Then 
\begin{enumerate}
   \item $V_{\mf{m}}(\lambda)$ is a simple $Z \rtimes G$-module;
   \item $V_{\mf{m}_1}(\lambda_1) \cong V_{\mf{m}_2}(\lambda_2)$ if and only if $\mf{m}_2 \in G \cdot \mf{m}_1$ and, moreover, if $g(\mf{m}_1) = \mf{m}_2$ then $\lambda_1 \cong \lambda_2$ via the conjugation isomorphism $g \colon G_{\mf{m}_1} \iso G_{\mf{m}_2}$;
   \item Every simple $Z \rtimes G$-module is isomorphic to $V_{\mf{m}}(\lambda)$ for some $\mf{m}$ and $\lambda$.
\end{enumerate}
Notice that $\dim V_{\mf{m}}(\lambda) = (|G| / |G_{\mf{m}}|) \dim \lambda$ since $Z \rtimes G$ is free of rank $|G| / |G_{\mf{m}}|$ over $Z \rtimes G_{\mf{m}}$. Therefore, $\dim V_{\mf{m}}(\lambda) \le |G|$ and $\dim V_{\mf{m}}(\lambda) = |G|$ if and only if $G_{\mf{m}} = \{ 1 \}$. 
\end{proof}

Recall from Section~\ref{sec:setupinvariant} that $(t,c) \in \mathfrak{c} =\C \oplus \Hom_W(\mc{S},\C)$, $R = \C[\mathfrak{c}]$ and $H$ denotes the rational Cherednik algebra defined over $R$. We wish to consider the case where $t = 0$ but $c$ is a variable, so we set $S = R /(t) \cong \C[c]$ and $H_S = H / t H$. The following lemma shows that the hypothesis of Corollary~\ref{cor:PIflatfamily} hold for $\eo H^{\Gamma}_S \eo$. 

\begin{lem}\label{lem:centresurject}
Let $c \in \Hom_W(\mc{S},\C)$. 
\begin{enumerate}
        \item[(i)] The algebra $H_{S}^{\Gamma}$ is a finite $Z(H_{S}^{\Gamma})$-module and $H_{c}^{\Gamma}$ is a finite $Z(H_{c}^{\Gamma})$-module. 
        \item[(ii)] The algebra $\eo H^{\Gamma}_S \eo$ is a prime affine $\C$-algebra, finite over $Z(\eo H^{\Gamma}_S \eo)$, and $\eo H^{\Gamma}_{c} \eo$ is a prime affine algebra, finite over $Z(\eo H^{\Gamma}_{c} \eo)$. 
        \item[(iii)] For all $\mf{m} \lhd S$ maximal the map $Z(\eo H^{\Gamma}_S \eo) \to Z(\eo H^{\Gamma}_S \eo / \langle \mf{m} \rangle)$ is surjective.
\end{enumerate}
\end{lem}

\begin{proof}
Part (i). Since $H_{S}$ is a finite $Z(H_{S})$-module, Lemma~\ref{lem:fginvariants} says that $H_{S}^{\Gamma}$ is a finite $Z(H_{S})^{\Gamma}$-module. Then $Z(H_{S})^{\Gamma} \subset Z(H_{S}^{\Gamma})$ implies that $H_{S}^{\Gamma}$ is finite over $Z(H_{S}^{\Gamma})$ too. The proof for $H^{\Gamma}_{c}$ is identical. 

Part (ii). Corollary~\ref{cor:primring} says that $\eo H^{\Gamma}_{c} \eo$ is prime. It also applies verbatim to $\eo H^{\Gamma}_S \eo$, showing that it is prime. By Theorem~\ref{glnpresentation} and Theorem~\ref{sobasis}, $H_S^{\Gamma}$ is an affine $S$-algebra; alternatively one can use that $\gr H_S^{\Gamma} \cong S \o ( P^{\Gamma} \rtimes W)$ is an affine $S$-algebra to deduce finite generation for $H_S^{\Gamma}$. Since $\eo$ is central in $H_S^{\Gamma}$, it follows that $\eo H^{\Gamma}_S \eo$ is an affine $S$-algebra. The fact that $S$ is a finitely generated $\C$-algebra then implies that $\eo H^{\Gamma}_S \eo$ is a finitely generated $\C$-algebra i.e. it is affine. This implies that $\eo H^{\Gamma}_{c} \eo$ is affine too since it is a quotient of $\eo H^{\Gamma}_S \eo$. The fact that these algebras are finite over their centres is a consequence of (i), using the decomposition of Lemma~\ref{lem:aidecomp}. 

Part (iii). If $\mf{m} \lhd S$ is the maximal ideal corresponding to the point $c$ then $\eo H^{\Gamma}_S \eo / \langle \mf{m} \rangle = \eo H^{\Gamma}_{c} \eo$. We have a commutative diagram 
\[
\begin{tikzcd}
    Z(\eo H^{\Gamma}_S \eo) \ar[d] \ar[r,"\sim"] & e H^{\Gamma}_S e \ar[d] \\
    Z(\eo H^{\Gamma}_{c} \eo) \ar[r,"\sim"] & e H^{\Gamma}_{c} e,
\end{tikzcd}
\]
where the horizontal arrows are isomorphisms by Corollary~\ref{cor:Satakeinv}. Since $H^{\Gamma}_S\to H^{\Gamma}_{c}$ is surjective and applying invariants, $h \mapsto e he$, is exact, the right vertical arrow is surjective. We deduce that the left vertical arrow is also surjective. 
\end{proof}

Let $Z = Z(\eo H^{\Gamma}_{c} \eo)$. If $M$ is a $\eo H^{\Gamma}_{c} \eo$-module then its support $\Supp \, M$ is the subset of $Y_{c}$ consisting of all prime ideals $\mf{p}$ in $Z \cong Z(H_c)^\Gamma$ such that $Z_{\mf{p}} \o_{Z} M \neq 0$. If $M$ is finitely generated then $\Supp \, M$ is a closed subset of $Y_{c}$. 

\begin{lem}\label{lemma-on-support}
Let $L$ be a simple $\eo H^{\Gamma}_{c} \eo$-module and $\mf{m} := \mr{ann}_Z L$. Then $\mf{m}$ is a maximal ideal and $\Supp \, L = \{ \mf{m} \}$.
\end{lem}

\begin{proof}
Note that $L$ is finite-dimensional over $\C$ by Lemma \ref{lem:centresurject}(ii) and \cite{BrownGoodearl}. 
 Schur's Lemma says that $\End_{\eo H^{\Gamma}_{c} \eo}(L) = \C$. The action of $Z$ on $L$ factors through $\End_{\eo H^{\Gamma}_{c} \eo}(L)$. Hence there exists a homomorphism $\chi \colon Z \to \C$ such that $z l = \chi(z) l $ for all $z \in Z$ and $l \in L$. In particular, $\mf{m} = \Ker \chi$ is a maximal ideal. The final claim then follows from the fact that the support of $L$ is the set of all prime ideals containing its annihilator; see e.g. \cite[II~Ex~5.6(b)]{Hartshorne} or \cite[Corollary~2.7]{Eisenbud}.   
\end{proof}

Thus, the support of a simple $\eo H^{\Gamma}_{c} \eo$-modules is a closed point of $Y_{c}$. 

\begin{lem}\label{lem:genericsimpleW}
Let $\mf{m}$ be a general (closed) point in $Y_c$, and let $L=\eo H_c^\Gamma e/\mf{m}\eo H_c^\Gamma e$ be the corresponding $\eo H_c^\Gamma \eo$-module. 
Then $L |_{\overline{W}} \cong \C \overline{W}$.  
\end{lem}

\begin{proof}
    Note that Lemma~\ref{lem:Wbaropenfree} says that the action of $\overline{W}$ on $V \git \, \Gamma = (\h \times \h^*) \git \, \Gamma$ is generically free. 
    Note also that $L$ can be represented as the induced module
$L \cong  \eo H^{\Gamma}_c e \o_{e H^{\Gamma}_{c} e} \chi$, where $\chi$ is a one-dimensional  $eH_c^\Gamma e$-module corresponding to the homomorphism
$$
e H^{\Gamma}_c e \to eH^{\Gamma}_c e/ e\mf{m}\cong \mathbb C.
$$
      Then the proof of \cite[Lemma~2.24]{EG} goes through verbatim. 
\end{proof}

\begin{thm}\label{thm:PIsimplemodules}
For all $c \in \Hom_W(\mc{S},\C)$, 
\begin{enumerate}
    \item[(i)] $H_{c}^{\Gamma}$ and $\eo H_{c}^{\Gamma}\eo$ are PI algebras and the PI degree of the prime PI-algebra $\eo H_{c}^{\Gamma}\eo$ equals $|\overline{W}|$.
    \item[(ii)] Let $L$ be a simple $\eo H_{c}^{\Gamma}\eo$-module whose support is contained in the regular locus of $Y_{c}$. Then $L |_{\overline{W}} \cong \C \overline{W}$. 
    \item[(iii)] The regular locus of $Y_{c}$ is contained in the Azumaya locus. 
\end{enumerate}
\end{thm}

\begin{proof}
Part (i). The fact that $H_{c}^{\Gamma}$ is a PI-algebra follows directly from Lemma~\ref{lem:centresurject}(i) and \cite[Corollary 13.1.3]{BrownGoodearlbook}. By Lemma~\ref{lem:centresurject}(iii), we may apply Corollary~\ref{cor:PIflatfamily} to $\eo H_{c}^{\Gamma}\eo$. In particular, the corollary implies that the PI-degree of $\eo H_{c}^{\Gamma}\eo$ is independent of $c$. Taking $c = 0$, it follows from Lemma~\ref{lem:Wbaropenfree} and Lemma~\ref{lem:PIdegreeskewgroup} that the PI-degree of $\eo H_{0}^{\Gamma}\eo = P^{\Gamma} \rtimes \overline{W}$ equals $|\overline{W}|$. 

Part (ii) and (iii) are an application of Theorem~\ref{thm:PIsimplemodulesA}, applied to $A = \eo H^{\Gamma}_{c} \eo$ and $e$ the symmetrizing idempotent. Note that (a) of Theorem~\ref{thm:PIsimplemodulesA} holds by Theorem~\ref{thm:doublecetralHGamma}, statement (b) is a consequence of the fact that $e H_c e$ is commutative (recall $t =0$) and Lemma~\ref{cor:CMmodule} says that $\eo H^{\Gamma}_{c} e$ is a Cohen-Macaulay $eH^{\Gamma}_{c} e$-module. Then all claims in (ii) and (iii) follow from Theorem~\ref{thm:PIsimplemodulesA}, except for the isomorphism 
\[
(A e /\mf{m} Ae) |_{\overline{W}} \cong \C \overline{W}. 
\]
Notice, in our case that the sheaf $\mc{M}$ on $Y_c$ considered in the proof of Theorem~\ref{thm:PIsimplemodulesA} is $\overline{W}$-equivariant, where $\overline{W}$ acts trivially on $Y_c$. It is a locally free sheaf when restricted to $Y_{c}^{\reg}$. Since the generic fiber is isomorphic to $\C \overline{W}$ by Lemma~\ref{lem:genericsimpleW} and $Y_{c}^{\reg}$ is connected, it follows that every fiber is isomorphic to $\C \overline{W}$.

\end{proof}

\begin{remark}
    Since $\eo H^{\Gamma}_{c} \eo$ need not have finite global dimension the usual arguments (see \cite[Lemma~3.3]{BrownGoodearl}) used to show that the Azumaya locus of $\eo H^{\Gamma}_{c} \eo$ is contained in the regular locus of $Y_{c}$ do not apply. We do not know if the inclusion holds. 
\end{remark}

Recall that the centre $Z := Z(\eo H^{\Gamma}_{c} \eo)$ of $\eo H^{\Gamma}_{c} \eo$ is a domain.
Let $F$ denote the field of fractions $Z$.  
By Posner's Theorem, the algebra $Q = F \o_{Z} (\eo H^{\Gamma}_{c} \eo)$ is a central simple algebra with centre equal to $F$. In fact, $Q$ is split over $F$. More precisely, we have the following corollary of the proof of Theorem~\ref{thm:PIsimplemodules}. 

\begin{cor}
    $Q \cong \mr{Mat}_{|\overline{W}|}(F)$. 
\end{cor}

\begin{proof}
    Choose a maximal ideal $\mf{m} \lhd Z$ lying in the Azumaya locus of $Y_c$. As shown in the proof of Theorems~\ref{thm:PIsimplemodules}, \ref{thm:PIsimplemodulesA} (see isomorphisms \eqref{localisation-matrix}), and by Theorem \ref{thm:doublecetralHGamma}
    \[
    (\eo H^{\Gamma}_{c} \eo)_{\mf{m}} = \End_{Z_{\mf{m}}}((\eo H^{\Gamma}_{c} e)_{\mf{m}}) \cong \mr{Mat}_{|\overline{W}|}(Z_{\mf{m}}),
    \]
    where we are using the Satake isomorphism of Corollary~\ref{cor:Satakeinv} to identify $Z$ with $e H^{\Gamma}_{c} e$. Then the corollary follows from the isomorphisms  
    \[
    Q = F \o_{Z} (\eo H^{\Gamma}_{c} \eo) = F \o_{Z_{\mf{m}}} (Z_{\mf{m}} \o_Z (\eo H^{\Gamma}_{c} \eo)) \cong F \o_{Z_{\mf{m}}} \mr{Mat}_{|\overline{W}|}(Z_{\mf{m}}) = \mr{Mat}_{|\overline{W}|}(F). 
    \]
\end{proof}

\section{The algebra of (quantum) Hamiltonian reduction}\label{sec:QHRtorus}

The action of $\T$ on the Calogero--Moser space $\Spec Z(H_{c})$ is Hamiltonian because the Poisson bracket is homogeneous of degree zero. Since $\{ \eu_{c}, f \} = \deg(f) f$ for any homogeneous $f \in Z(H_{c})$, the (co-)moment map $\mu^* \colon {\rm Sym} (\mr{Lie} \, \T) \to Z(H_{c})$ is given by $\mu^*(1) = \eu_{c}$ after an identification ${\rm Lie} \T \cong \C\ni 1$. Therefore, the Hamiltonian reduction of $\Spec Z(H_{c})$ is 
$$
\mu^{-1}(\zeta) \git \, \T = \Spec \left( Z(H_{c})^{\T} / \langle \mr{eu}_{c} - \zeta \rangle \right), \quad \forall \, \zeta \in \C = (\mr{Lie} \, \T)^*.
$$

\begin{defn}
   For $\zeta \in \C$, the algebra of (quantum) Hamiltonian reduction is 
   $$
   A_{t,c,\zeta} = A_{t,c,\zeta}(W) = \eo H_{t,c}^{\mf{gl}(n)}\eo / \langle \eo \mr{eu}_{c} - \zeta \rangle.
   $$
\end{defn}

\begin{remark}
    The action of $\SL_2$ on $\Spec Z(H_{c})$ is also Hamiltonian and one can consider the (quantum) Hamiltonian reduction of $H_{t,c}$ or $\Spec Z(H_{c})$, as above. We hope to return to this construction in future work. 
\end{remark}

Recall that $P = \C[\h \times \h^*]$. Since $\eo \eu_{c}$ is central (and hence normal) in $A_{t,c,\zeta}$ and its symbol $\mr{eu}_0$ is regular in $P^{\T} \rtimes W$, Lemma~\ref{lem:assgrregularelement} immediately implies the following. 

\begin{lem}\label{lem:assgrAW}
	We have 
	$$
	\gr A_{t,c,\zeta} = (P^{\T} / \langle \mr{eu}_0 \rangle ) \rtimes \overline{W}, \quad \gr (Z(H_{c})^{\T} /  \langle \mr{eu}_{c} - \zeta \rangle ) \cong (P^{\T} / \langle \mr{eu}_0 \rangle )^W.
	$$
\end{lem}

\begin{lem}\label{lem:centre0Hamred}
The centre of $(P^{\T} / \langle \mr{eu_0} \rangle ) \rtimes \overline{W}$ equals $(P^{\T} / \langle \mr{eu_0} \rangle )^W$.
\end{lem}

\begin{proof}
    The proof is similar to that of Proposition~\ref{prop:centreskewgroup1}. The key point being that $P^{\T} / \langle \mr{eu_0} \rangle$ is a domain on which $\overline{W}$ acts faithfully. 
\end{proof}

If $S$ is not a domain then the centre of $S \rtimes W$ can properly contain $S^W$ even if $W$ acts faithfully on $S$. 

The affine variety $(\h \times \h^*)/W$ is a Poisson variety, where the Poisson bracket on $P^W$ comes by restriction from the standard bracket on $P$. By definition of Hamiltonian reduction,  
$$
\Spec (P^{\T} / \langle \mr{eu}_0 \rangle )^W
$$
is again a Poisson scheme. 

We recall that Beauville \cite{Beauville} introduced the notion of symplectic singularities. A normal irreducible variety $X$ over $\C$ is said to have {\it symplectic singularities} if the smooth locus $U$ of $X$ carries an algebraic symplectic form $\omega$ such that, for any resolution of singularities $\pi \colon Y \to X$, $\pi^* \omega$ extends to a regular $2$-form on $Y$. 

The following is well-known to experts, but we sketch the proof since we do not know of a suitable reference. 

\begin{lem}\label{lem:sympsingGorenstein}
	Let $X$ have symplectic singularities. Then
	\begin{enumerate}
		\item[(i)] $X$ has rational singularities,  
		\item[(ii)] is Gorenstein, and hence
		\item[(iii)] is Cohen-Macaulay.
	\end{enumerate} 
\end{lem}

\begin{proof}
Let $n = (1/2) \dim X$. If $\omega \in \Omega^2_{U}$ is the symplectic form then $s = \wedge^n \omega$ is a non-vanishing top form. It trivializes $\Omega_{U}$ and hence implies that the canonical sheaf $\omega_X = j_*  \Omega_{U}$ is trivial, where $j \colon U \hookrightarrow X$. 

Let $f \colon Y \to X$ be a resolution of singularities which is an isomorphism over $U$. Let $K_X$ and $K_Y$ denote the canonical divisor on $X$ and $Y$ respectively. Since $\omega$ extends to a regular two form on $Y$, $s$ also extends to a regular form on $Y$. Therefore, in the equality
\[
K_Y = f^*(K_X) + \sum_i m_i E_i, 
\]
	the coefficients $m_i$ are non-negative since they are the order of vanishing of $f^* s$ along the exceptional divisor $E_i$. Here the sum is over all exceptional divisors of $f$. This means that $X$ has canonical singularities. By \cite[Theoreme~1]{ElkikCanonical} this means that $X$ has rational singularities. In particular, Kempf's Theorem \cite{Toroidalembeddings} says that $X$ is Cohen-Macaulay. This implies that the dualizing complex $\omega_X^{\idot}$ of $X$ is quasi-isomorphic to $\omega_X \cong \mc{O}_X$. Thus, $\mc{O}_X$ has finite injective dimension i.e. $X$ is Gorenstein.	 
\end{proof}


The variety $X$ with symplectic singularities is said to be {\it a conic symplectic singularity} if it is affine, $\C[X]$ is $\N$-graded with $\C[X]_0 = \C$ and the form $\omega$ is homogeneous of weight $\ell > 0$. Examples of conic symplectic singularities include the normalization of a nilpotent orbit closure in a semi-simple Lie algebra and symplectic quotient singularities. 

Let $\mc{O}_{\mr{min}} \subset \mf{g}=\mf{sl}(n)$ denote the minimal nilpotent orbit (consisting of all nilpotent rank one matrices). Recall that it is a symplectic variety. Let $(-,-)$ be the Killing form on $\mf{g}$. The tangent space to the orbit at a point $x \in \mc{O}_{\mr{min}}$ can be represented as $T_x \mc{O}_{\mr{min}} \cong \mf{g}/\mf{g}_x$, where $\mf{g}_x=\{y\in \mf{g} | [x,y]=0\}$. Then the Kirillov--Kostant--Souriau  symplectic form  is given by  $\omega_x(\bar y, \bar z) := ([y,z], x)$, where $\bar y, \bar z \in \mf{g}/\mf{g}_x$ are images of $y, z \in \mf{g}$, respectively.

The closure $\overline{\mc{O}}_{\mr{min}}$ is normal and hence is a conic symplectic singularity; see \cite[(2.6)]{Beauville}. A subgroup $W \subset \GL(n)$ acts on $\overline{\mc{O}}_{\mr{min}}$ by conjugation. Since the Killing form is invariant under such an action, the symplectic form on the orbit is preserved as well, and hence $\overline{\mc{O}}_{\mr{min}}/W$ is again a (conic) symplectic singularity by \cite[Proposition~2.4]{Beauville}. 

\begin{lem}\label{lem:isoHamredminimialnilp}
Consider the Hamiltonian action of the torus $\T$ on the cotangent bundle $T^*\mf{h}\cong \mf{h}\oplus\mf{h}^*$ such that the weight of the action on $\mf{h}$ is $1$, and the weight of the action on $\mf{h}^*$ is $-1$.
\begin{enumerate}
\item[(i)] The closure of the minimal nilpotent orbit is the Hamiltonian reduction: $\overline{\mc{O}}_{\mr{min}}\cong \nu^{-1}(0)\git \, \T$ as Poisson schemes, where $\nu\colon T^*\mf{h} \to \C\cong \mr{Lie}(\T)^*$ is the moment map. 
\item[(ii)] There is an isomorphism of Poisson schemes $\Spec (P^{\T} / \langle \mr{eu}_0 \rangle )^W \cong \overline{\mc{O}}_{\mr{min}}/W$. 
    In particular, $(P^{\T} / \langle \mr{eu}_0 \rangle )^W$ is an integrally closed domain. 
\end{enumerate}
\end{lem}

\begin{proof}
    The moment map $\nu\colon \mf{h}\oplus\mf{h}^*\to \C$ sends a pair $(v, \lambda)$ to $\lambda(v)$.
    Consider also the map $\varphi \colon \h \oplus \h^* \to \End_{\C}(\h)$, $(v,\lambda) \mapsto v \o \lambda$. The image of $\varphi$  is the space of all rank at most one endomorphisms. Furthermore, the induced map from 
    $(\mf{\h} \oplus \h^*)\git \T$ to the space of endomorphisms of $\mf{h}$ of rank at most $1$ is one to one. Such an endomorphism is nilpotent if and only if its trace $ \lambda(v)=0$, which implies that $\overline{\mc{O}}_{\mr{min}}\cong \nu^{-1}(0)\git \T$.


The equation on $\h\oplus\h^*$  defined by the polynomial $\mr{eu}_0$ is precisely the condition $\lambda(v) = 0$. This means that the set $\nu^{-1}(0) = V(\mr{eu}_0)$. 
Hence $\Spec(P^T/ \langle \mr{eu}_0 \rangle) \cong \overline{\mc{O}}_{\mr{min}}$. Since this isomorphism is $W$-equivariant, it descends to 
    $$
    \Spec (P^{\T} / \langle \mr{eu}_0 \rangle )^W = \left(\nu^{-1}(0) \git \, \T \right) /W \cong \overline{\mc{O}}_{\mr{min}}/W.
    $$
    In particular, the left hand side is a normal domain. 
\end{proof}

\begin{lem}\label{lem:opencodim2freeOmin}
    The group $\overline{W}$ acts freely on a dense open subset of $\overline{\mc{O}}_{\mr{min}}$ with complement of codimension at least two. 
\end{lem}

\begin{proof}
    Presumably this can be checked directly, however we will deduce it from the fact that $\overline{W}$ preserves the Poisson structure on $\overline{\mc{O}}_{\mr{min}}$. 
    
    Since $\overline{\mc{O}}_{\mr{min}}$ is normal, its smooth locus $U$ has complement of codimension at least two; in fact, $\overline{\mc{O}}_{\mr{min}}$ has an isolated singularity. Therefore it suffices to show that $\overline{W}$ acts freely on an open subset of $U$ with complement of codimension at least two. Let $u \in U$ and $\overline{W}_u$ the stabilizer of $u$ in $\overline{W}$. Then $\overline{W}_u$ acts on the tangent space $T_u U$ and it suffices to show that this action factors through $\SL(T_u U)$. But $T_u U$ is a symplectic vector space and the action of $\overline{W}_u$ preserves the symplectic form. Thus, the action of $\overline{W}_u$ factors through $\mr{Sp}(T_u U)$.

    Finally, we need to check that there is at least one point $A$ in $\overline{\mc{O}}_{\mr{min}}$ whose stabilizer $\overline{W}_A$ is trivial. Note that each $A \in \mc{O}_{\mr{min}}$ is characterized by the fact that there exists $0 \neq v \in \h$ such that $\mr{Im} \, A = \C \cdot v$ and $A(v) = 0$. Choose a non-zero $v \in \h$ such that $v$ is not a eigenvector of any $w \in W \setminus Z(W)$ and let $A \in \mc{O}_{\mr{min}}$ with $\mr{Im} \, A = \C \cdot v$. Then $w A w^{-1} = A$ implies that $v$ is an eigenvector of $w$. Thus, $\overline{W}_A = \{ 1 \}$.   
\end{proof}

\begin{prop}
	The algebra $A_{t,c,\zeta}$ is a prime noetherian Auslander--Gorenstein ring. It has Gelfand--Kirillov dimension $2 (\dim \h - 1)$.
\end{prop}

\begin{proof}
	As noted above, the variety $\overline{\mc{O}}_{\mr{min}}$ has symplectic singularities. By Lemma~\ref{lem:sympsingGorenstein}, this implies that the ring $\C[\overline{\mc{O}}_{\mr{min}}]$ is Gorenstein. Thus, $P^{\T} / \langle \mr{eu}_0 \rangle$ is a noetherian Gorenstein domain. Therefore it is Auslander--Gorenstein.
 
 This implies that $(P^{\T} / \langle \mr{eu}_0 \rangle ) \rtimes \overline{W}$ is a noetherian Auslander--Gorenstein ring \cite{YiGlasgow}. It is prime by \cite[Corollary~12.6]{Passman}. Hence Lemma~\ref{lem:assgrAW} implies that $A_{t,c,\zeta}(W)$ is also Auslander--Gorenstein (\cite[Theorem 3.9]{BjorkAuslander}) and prime. Finally, we note that $\dim \overline{\mc{O}}_{\mr{min}} = 2(\dim \h - 1)$, which implies that the Gelfand--Kirillov dimension of $A_{t,c,\zeta}(W)$ is also $2 (\dim \h - 1)$; see \cite[Proposition~8.1.14]{MR}.    
\end{proof}

Let $A$ be a filtered algebra such that $\gr_{\mc{F}} A$ is a commutative algebra. Then $\gr_{\mc{F}} A$ is an $\N$-graded Poisson algebra. If $B$ is an $\N$-graded Poisson algebra, then we say that $A$ is a \textit{filtered quantization} of $B$ if there exists a graded Poisson isomorphism $\gr_{\mc{F}} A \cong B$. 

Recall that $e$ is the symmetrizing idempotent in $\C W$. 

\begin{thm}\label{thm:filteredquantofA0}
    Suppose $t \neq 0$. Then the algebra $e A_{t,c,\zeta} e$ is a filtered quantization of $\overline{\mc{O}}_{\mr{min}}/W$. 
\end{thm}

\begin{proof}
    It is a consequence of the proof of \cite[Theorem~1.6]{EG} (in particular, Claim 2.25) that $e H_{t,c} e$ is a filtered quantization of $P^W$, assuming $t \ne 0$. Since the filtration on $e H_{t,c}^{\T} e$ is by restriction and $P^{\T \times W}$ is a Poisson subalgebra of $P^W$ this implies that $e H_{t,c}^{\T} e$ is a filtered quantization of $P^{\T \times W}$. Finally, since the filtration on $e A_{t,c,\zeta}(W) e$ is the quotient filtration coming from the filtration on $e H_{t,c}^{\T} e$ and the Poisson structure on $(P^{\T} / \langle \mr{eu}_0 \rangle )^W$ comes from the fact that $\langle \mr{eu}_0 \rangle^W$ is a Poisson ideal in $P^{\T \times W}$, we deduce that $e A_{t,c,\zeta}e$ is a filtered quantization of $\overline{\mc{O}}_{\mr{min}}/W$ by applying the (Poisson) isomorphism of Lemma~\ref{lem:isoHamredminimialnilp}. 
\end{proof}

\begin{thm}\label{thm:finiteglobaldimangular}
	For Weil generic $(t,c,\zeta)$, $A_{t,c,\zeta}(W)$ has finite global dimension. 
\end{thm}

\begin{proof}
We assume $t = 1$. Then the claim follows from the fact that localization holds for Weil generic parameters. We use here the notion of sheaves of Cherednik algebras on $\mathbb{P}(\h)$, as studied in \cite{ProjAffine}. By \cite[Lemma~5.4.1]{ProjAffine}, the sheaf $\mc{H}_{\zeta,c}(\mathbb{P}(\h),W)$ of Cherednik algebras on the projective space $\mathbb{P}(\h)$ has global sections
\[
\Gamma(\mathbb{P}(\h),\mc{H}_{\zeta,c}(\mathbb{P}(\h),W)) = H_{1,c}^{\T} / \langle \eu_{c} - \zeta \rangle .
\]
This implies that $A_{1, c,\zeta}$ is a direct summand of $H (\mathbb{P}) := \Gamma(\mathbb{P}(\h),\mc{H}_{\zeta,c}(\mathbb{P}(\h),W))$ and hence it suffices to show that $H (\mathbb{P})$ has finite global dimension for Weil generic $(c,\zeta)$. 

The stalks of $\mc{H}_{\zeta,c}(\mathbb{P}(\h),W)$ have finite global dimension (bounded by $2 \dim \mathbb{P}(\h)$). This implies that the sheaf $\mc{E}xt^q_{\mc{H}}(\mc{M},\mc{N})$ is zero for $q > 2 \dim \mathbb{P}(\h)$ if $\mc{M},\mc{N}$ are coherent over $\mc{H}_{\zeta,c}(\mathbb{P}(\h),W)$.  The Grothendieck spectral sequence 
 $$
 E_2^{p,q} = H^p(\mathbb{P}(\h),\mc{E}xt^q_{\mc{H}}(\mc{M},\mc{N})) \Rightarrow \Ext_{\mc{H}}^{p+q}(\mc{M},\mc{N})
 $$
 converges because $H^p(\mathbb{P}(\h),\mc{F}) = 0$ for $p > 2 \dim_{\C} \mathbb{P}(\h)$ by Grothendieck's \cite[Theorem~2.7]{Hartshorne}  for any sheaf $\mc{F}$ on $\mathbb{P}(\h)$. In particular, since the second page is bounded we deduce that $\Ext_{\mc{H}}^{k}(\mc{M},\mc{N}) = 0$ for $k$ sufficiently large (independent of $\mc{M},\mc{N}$). By \cite[Theorem~1.2.1]{ProjAffine}, the global sections functor $\Gamma \colon \Lmod{\mc{H}_{\zeta,c}(\mathbb{P}(\h),W)} \to \Lmod{H (\mathbb{P})}$ from the category of coherent $\mc{H}_{\zeta,c}(\mathbb{P}(\h),W)$-modules to finitely generated modules over $H (\mathbb{P})$ is an equivalence (i.e. affinity holds) for Weil generic $(c, \zeta)$. If we assume affinity then 
 \[
 \Ext_{\mc{H}}^{i}(\mc{M},\mc{N}) = \Ext_{H (\mathbb{P})}^{i}(\Gamma(\mathbb{P}(\h),\mc{M}),\Gamma(\mathbb{P}(\h),\mc{N}))
 \]
 for any pair of coherent $\mc{H}_{c}(\mathbb{P}(\h),W)$-modules $\mc{M},\mc{N}$. It follows that $\Ext_{H (\mathbb{P})}^{k}(M,N) = 0$ (again, independent of $M,N$) for $k$ sufficiently large and all finitely generated $M,N$. This implies that $H (\mathbb{P})$ has finite global dimension. 
 
Compare with \cite[VI,~Theorem~1.10(ii)]{BorelDmod}, which is the proof in the case of $\dd$-modules. 
\end{proof}

Note that it is not true that $A_{t,c,\zeta}$ always has finite global dimension. For instance, if $\dim \h > 1$ then take $t = 1$ and $c = 0$. If $A_{1,0,\zeta} = \dd_{\zeta}(\mathbb{P}(\h)) \rtimes \overline{W}$ has finite global dimension then so too would the ring $\dd_{\zeta}(\mathbb{P}(\h))$ of global $\zeta$-twisted differential operators on $\mathbb{P}(\h)$. But it is known \cite{Sheldonglobadim} that when $\zeta = -1,-2,\ds,-(\dim \h - 1)$ this is not the case.

\subsection{The double centralizer property}

Consider the space $A_{t,c,\zeta}e$ as a left $A_{t,c,\zeta}$-module and as a right $e A_{t,c,\zeta}e$-module. The next theorem establishes the double centralizer property for the bimodule $A_{t,c,\zeta}e$. 

\begin{thm}\label{thm:doublecetralH0}
We have an isomorphism 
    $\End_{e A_{t,c,\zeta} e}(A_{t,c,\zeta} e) \cong A_{t,c,\zeta}$. 
\end{thm}

\begin{proof}
First, we note that $A_{0,0,0} = \C[\overline{\mc{O}}_{\mr{min}}] \rtimes \overline{W}$. Since $\overline{\mc{O}}_{\mr{min}}$ is a normal variety and the group $\overline{W}$ acts freely on a dense open subset of $\overline{\mc{O}}_{\mr{min}}$ with complement of codimension at least two (Lemma~\ref{lem:opencodim2freeOmin}), the double centralizer property holds for $A_{0,0,0}$ by \cite[Proposition~E.2.2(2)]{BonnafeRouquier}. Therefore the proof of \cite[Theorem~1.5(iv)]{EG} shows that it holds for all parameters $(t,c,\zeta)$. 
\end{proof}

As for the usual rational Cherednik algebra, we say that the parameters $(t,c,\zeta)$ are \textit{aspherical} if there exists a nonzero
$A_{t,c,\zeta}$-module $M$ such that 
$e M = 0$. The parameters $(t,c,\zeta)$ are called
\textit{spherical} if they are not aspherical.
 The following is an immediate corollary of Theorem~\ref{thm:doublecetralH0}, as explained in \cite[Corollary~1.6.3]{BellamySRAlecturenotes}.

\begin{cor}
The algebra $A_{t,c,\zeta}$ is Morita equivalent to $e A_{t,c,\zeta}e$ if and only if the parameters are spherical. 
\end{cor}

\subsection{The centre of $A_{t,c,\zeta}(W)$}
Similarly to Section \ref{doublecsection}, we have the following corollary from Theorem \ref{thm:doublecetralH0}. 

\begin{cor}\label{cor:Satakequot}
The map $z \mapsto e z$ is an isomorphism $Z(A_{t,c,\zeta}) \iso Z(e A_{t,c,\zeta} e)$.  
\end{cor}

We deduce that:

\begin{cor}\label{cor:centreHamredT}
    The centre of $A_{0,c,\zeta}$ is isomorphic to $Z(H_{c})^{\T} / \langle \mr{eu}_{c} - \zeta \rangle$.
\end{cor}

\begin{proof}
    At $t = 0$, the algebra $e H_{c}^{\T} e$ is commutative since $e H_{c} e$ is commutative. Therefore, the quotient $e A_{0,c,\zeta}e$ is commutative in this case. Corollary~\ref{cor:Satakequot} then implies that $z \mapsto e z$ is an isomorphism $Z(A_{0,c,\zeta}) \iso e A_{0,c,\zeta} e$. The corollary will follow if we can show that the composite morphism
    \[
    Z(H_{c})^{\T} / \langle \mr{eu}_{c} - \zeta \rangle \to Z(A_{0,c,\zeta}) \to e A_{0,c,\zeta} e = eH_{c}^{\T} e / \langle e (\mr{eu}_c - \zeta) \rangle 
    \]
    is an isomorphism. But this is a direct consequence of the fact that the isomorphism $  Z(H_{c})^{\T} \iso eH_{c}^{\T} e$ (which is the restriction of the usual Satake isomorphism $Z(H_{c}) \iso e H_{c} e$ to $\T$-invariant elements), given by $z \mapsto e z$, maps $\langle \mr{eu}_{c} - \zeta \rangle$ to $\langle e (\mr{eu}_c - \zeta) \rangle$
\end{proof}

\begin{lem}\label{lem:symplsingCasimir}
    Let $X$ be an irreducible affine variety with symplectic singularities. Then $\mr{Cas}(\mc{O}(X)) = \C$. 
\end{lem}

\begin{proof}
Assume that there exists $t \in \mr{Cas}(\mc{O}(X)) \setminus \C$. Since $\mc{O}(X)$ is a domain, the morphism $\C[t] \to \mc{O}(X)$ is an embedding and hence defines a dominant morphism $p \colon X \to \mathbb{A}^1$. Every non-empty fiber of $p$ defines a closed Poisson subvariety of $X$. Each of these fibers is therefore a union of symplectic leaves. Since $p$ is dominant there are infinitely many non-zero fibers and hence $X$ contains infinitely many symplectic leaves. However, by \cite[Theorem~2.5]{Kaledinsympsingularities} the variety $X$ is a holonomic Poisson variety, and it has only finitely many symplectic leaves by \cite[Proposition~3.1]{Kaledinsympsingularities}. This is a contradiction. Hence $\mr{Cas}(\mc{O}(X)) = \C$.
\end{proof}

As a consequence of the Satake isomorphism, we deduce that:

\begin{cor}
	If $t \neq 0$ then the centre of $A_{t,c,\zeta}$ is $\C$. 
\end{cor}

\begin{proof}
By Corollary~\ref{cor:Satakequot}, it suffices to show that the centre of $e A_{t,c,\zeta}e$ is trivial. By Theorem~\ref{thm:filteredquantofA0}, this algebra is a filtered quantization of $\C[\overline{\mc{O}}_{\mr{min}}/W]$. We give the centre of $e A_{t,c,\zeta}e$ a filtration by restriction. Then $\gr Z(e A_{t,c,\zeta} e)$ is contained in $\mr{Cas}(\C[\overline{\mc{O}}_{\mr{min}}]^W)$. Since $\overline{\mc{O}}_{\mr{min}}/W$ has symplectic singularities, Lemma~\ref{lem:symplsingCasimir} says that $\mr{Cas}(\C[\overline{\mc{O}}_{\mr{min}}]^W)= \C$. This forces $\gr Z(e A_{t,c,\zeta} e) = \C$, and hence $Z(e A_{t,c,\zeta}e) = \C$ too.  
\end{proof}

Fixing $t = 0$ but considering $c,\zeta$ as formal variables, we get an algebra $A_0$. Then $\Spec Z(A_{0})$ is a flat family over $\mf{\tilde c} = \Hom_W(\mc{S},\C) \oplus \C$. We grade $\mf{\tilde c}$ by giving it degree $2$. In this way, $A_{0}$ is a graded $\C[\mf{\tilde c}]$-algebra.  

\begin{thm}\label{sympsing}
	$\Spec Z(A_{0}) \to \mf{\tilde c}$ is a graded Poisson deformation of the symplectic singularity $\overline{\mc{O}}_{min}/W$.  
\end{thm}

\begin{proof}
We consider $H_S$ as a $\C[\mf{\tilde c}]$-algebra, where the defining relations do not involve $\zeta$. Then $Z(H_S)$ is a flat graded Poisson deformation of $P^W$ over $\mf{\tilde c}$. This implies that $Z(\eo H_S^{\T} \eo) \cong Z(H_S)^{\T}$ is a graded Poisson deformation of $P^{W \times \T}$. Then $\eo \eu_{c} - \zeta$ generates a graded Poisson ideal in $Z(\eo H_S^{\T} \eo)$ such that the quotient 
\[
Z(\eo H_S^{\T} \eo) / \langle \eo \eu_{c} - \zeta \rangle
\]
is flat over $\C[\mf{\tilde c}]$ by Lemma~\ref{lem:assgrAW}. 

Finally, we note that the fibre of this map above $0 \in \mf{\tilde c}$ is the centre of $A_{0,0,0} = (P^{\T} / \langle \mr{eu} \rangle ) \rtimes \overline{W}$. By Lemma~\ref{lem:centre0Hamred}, this equals the ring of functions on the quotient $\overline{\mc{O}}_{min}/W$.
\end{proof}

\begin{remark}
    As follows from Namikawa’s results \cite{Namikawa}, the conic symplectic singulartity $\overline{\mc{O}}_{min}/W$ admits a universal graded Poisson deformation $\mc{X} \to \mf{b}$. Therefore, Theorem~\ref{sympsing} implies that there is a unique graded map $\mf{\tilde c} \to \mf{b}$ such that 
    \[
    \Spec Z(A_{0}) \cong \mf{\tilde c} \times_{\mf{b}} \mc{X}
    \]
    as Poisson varieties. We expect that one can use arguments as in \cite{BellamyNamikawa} to explicitly describe the map $\mf{\tilde c} \to \mf{b}$.     
\end{remark}

\small{

\bibliography{biblo}{}
\bibliographystyle{plain} }

\end{document}